\documentclass[12pt,reqno]{amsart}

\usepackage{amsmath, amssymb, amsthm}
\usepackage[colorlinks=true]{hyperref}
\usepackage[alphabetic]{amsrefs}
\usepackage{verbatim}
\usepackage{xcolor}
\usepackage{url}
\usepackage{enumitem}

\newtheorem{lemma}{Lemma}[section]
\newtheorem{theorem}[lemma]{Theorem}

\newtheorem{proposition}[lemma]{Proposition}
\theoremstyle{definition}
\newtheorem{definition}[lemma]{Definition}

\newtheorem{example*}[lemma]{Example}
\newtheorem{remark}[lemma]{Remark}

\theoremstyle{remark}

\makeatletter
\newtheorem*{rep@theorem}{\rep@title}
\newcommand{\newreptheorem}[2]{%
\newenvironment{rep#1}[1]{%
 \def\rep@title{{\bf #2 \ref{##1}}}%
 \begin{rep@theorem}}%
 {\end{rep@theorem}}}
\makeatother
\newreptheorem{theorem}{Theorem}

\DeclareRobustCommand{\qedify}[1]{%
  \ifmmode \quad\hbox{#1}
  \else
    \leavevmode\unskip\penalty9999 \hbox{}\nobreak\hfill
    \quad\hbox{#1}%
  \fi
}
\newenvironment{example}{\begin{example*}\pushQED{\qedify{$\diamondsuit$}}}{\popQED\end{example*}}

\numberwithin{equation}{section}

\newcommand{\K}{\ensuremath{\Bbbk}} 

\newcommand{\Rbar}{\ensuremath{\overline{\mathbb R}}}
\newcommand{\Trop}{\ensuremath{\mathcal{T}\!rop}}

\DeclareMathOperator{\Hom}{Hom}
\DeclareMathOperator{\inn}{in}

\DeclareMathOperator{\val}{val}
\DeclareMathOperator{\im}{im}
\DeclareMathOperator{\trop}{trop}
\DeclareMathOperator{\supp}{supp}

\DeclareMathOperator{\spann}{span}
\DeclareMathOperator{\mult}{mult}
\DeclareMathOperator{\GL}{GL}

\DeclareMathOperator{\Gr}{Gr}

\newcommand{\PuiseuxC}{\ensuremath{\mathbb C \{\!\{t \} \! \}}}
\newcommand{\mon}{M}

\newcommand{\M}{\mathcal{M}}

\makeatletter
\newcommand{\superimpose}[2]
{{\ooalign{$#1\@firstoftwo#2$\cr\hfil$#1\@secondoftwo#2$\hfil\cr}}}
\makeatother
\newcommand{\ttimes}{\hspace{0.4mm}{\mathpalette\superimpose{{\circ}
{\cdot}}}\hspace{0.4mm}}
\newcommand{\tplus}{\mathrel{\oplus}}
\newcommand{\bigtplus}{\bigoplus}

\begin{document}

\title{Tropical schemes, tropical cycles, and valuated matroids}

\author{Diane Maclagan}
\address{Mathematics Institute, University of Warwick, Coventry CV4 7AL, United 
Kingdom.}
\email{D.Maclagan@warwick.ac.uk}

\author{Felipe Rinc\'on}
\address{Department of Mathematics, University of Oslo, 0851 Oslo, Norway.}
\email{feliperi@math.uio.no}

\begin{abstract}
We show that the weights on a tropical variety can be recovered from
the tropical scheme structure proposed in
\cite{Giansiracusa2}, so there is a well-defined Hilbert-Chow morphism
from a tropical scheme to the underlying tropical cycle.  For a
subscheme of projective space given by a homogeneous ideal $I$ we show
that the Giansiracusa tropical scheme structure contains the same
information as the set of valuated matroids of the vector spaces
$I_d$ for $d \geq 0$. We also give a combinatorial criterion to
determine whether a given relation is in the congruence defining the
tropical scheme structure.
\end{abstract}

\maketitle

\section{Introduction}

The tropicalization of a subvariety $Y$ in the $n$-dimensional
algebraic torus $T$ is a polyhedral complex $\trop(Y)$ that is a
``combinatorial shadow'' of the original variety.  Some invariants of
$Y$, such as the dimension, are encoded in $\trop(Y)$.  The complex
$\trop(Y)$ comes equipped with positive integer weights on its
top-dimensional cells, called multiplicities, that make it into a {\em
  tropical cycle}.  This extra information encodes information about the
intersection theory of compactifications of the original variety $Y$;
see for example \cite{KatzPayne}.

In \cite{Giansiracusa2} the authors propose a notion of {\em tropical
  scheme structure} for tropical varieties, which takes the form of a
congruence on the semiring of tropical polynomials (see Section
\ref{s:preliminaries}).  When $Y \subset T$ is a subscheme defined by
an ideal $I \subset K[x_1^{\pm 1},\dots,x_n^{\pm 1}]$ this congruence
is denoted by $\Trop(I)$. In \cite{Giansiracusa2} the tropical scheme
structure is defined in the slightly more general context of $\mathbb
F_1$-schemes.

In this paper we investigate the relation between these tropical
schemes, ideals in the semiring of tropical polynomials, and
the theory of valuated matroids introduced by Dress and Wenzel
\cite{DressWenzel}.  
We also show that the tropical cycle of a scheme can be reconstructed
from the corresponding congruence.

Our first result is the following.
\begin{theorem} \label{t:equivalent}
Let $K$ be a field with a valuation $\val \colon K \rightarrow \Rbar :=
\mathbb R \cup \{ \infty \}$, and let $Y$ be a subscheme of $T \cong
(K^*)^n$ defined by an ideal $I \subset K[x_1^{\pm 1},\dots,x_n^{\pm
    1}]$.  Then any of the following three objects determines the others:
\begin{enumerate}
\item \label{item:congruence} The congruence $\Trop(I)$ on the
  semiring $S:=\Rbar[x_1^{\pm  1},\dots,x_n^{\pm 1}]$ of tropical Laurent polynomials;
\item \label{item:ideal} The ideal $\trop(I)$ in $S$;
\item \label{item:valuated} The set of valuated matroids of the vector spaces $I^h_d$, 
  where $I^h \subset K[x_0,\dots,x_n]$ is the homogenization of the ideal $I$, and
  $I^h_d$ is its degree $d$ part.
\end{enumerate}
\end{theorem}
When the valuation on $K$ is trivial, this says that the tropical
scheme structure $\Trop(I)$ is equivalent to the information of the supports of
all polynomials in $I$, and also of the (standard) matroids of the vector
spaces $I^h_d$.  Theorem \ref{t:equivalent} is mostly proved in
Section~\ref{s:preliminaries}, though we postpone the discussion of
valuated matroids, including recalling their definition, to
Section~\ref{s:matroids}.  The version proved there
(Theorem~\ref{t:generalizedEquivalence}) also holds for a subscheme $Z
\subset \mathbb P^n$ given by a homogeneous ideal in
$K[x_0,\dots,x_n]$.

In Section~\ref{s:multiplicity} we show that the tropical cycle
structure on a tropical variety $\trop(Y)$ can be recovered from its
tropical scheme structure, answering the question raised in
\cite{Giansiracusa2}*{Remark 7.2.3}.

\begin{theorem} \label{t:mainthm}
Let $Y \subset T$ be a subscheme defined by an ideal $I \subset
K[x_1^{\pm 1},\dots,x_n^{\pm 1}]$.  The multiplicities of the maximal
cells in the tropical variety $\trop(Y)$ can be recovered from the
congruence $\Trop(I)$.
\end{theorem}

The classical Hilbert-Chow morphism takes a subscheme of $\mathbb P^n$
to the associated cycle in the Chow group of $\mathbb P^n$.
Theorem~\ref{t:mainthm} can thus be thought of as a tropical version
of this morphism.  

Finally, in Section~\ref{s:matroids} we investigate in more depth 
the structure of the congruence $\Trop(I)$, 
and use ideas from valuated matroids and tropical linear spaces to characterize when
a relation lives in $\Trop(I)$.
We also show that any tropical polynomial has a distinguished representative in its
equivalence class in $\Trop(I)$, and give a combinatorial
procedure to compute it.

\vspace{2mm}
\noindent {\bf Acknowledgments.} Both authors were partially
supported by the EPSRC grant EP/I008071/1. We thank Florian Block for
helpful conversations about matroids, and Jeff Giansiracusa for
discussion on \cite{Giansiracusa2} and comments on an earlier draft of
this paper.  The first author also thanks the Max Planck Institute
for Mathematics for hospitality while some of this paper was written.

\section{Tropical varieties and their scheme structure}
\label{s:preliminaries}

In this section we recall the necessary background on tropical
geometry and the definition of the tropical scheme structure proposed in
\cite{Giansiracusa2}.  We also develop some fundamental properties of
these congruences, leading to part of the proof of
Theorem~\ref{t:equivalent}.

Throughout this paper we denote by $\Rbar$ the tropical semiring
(or min-plus algebra) 
\[
\Rbar := (\mathbb R \cup \{ \infty \}, \tplus, \ttimes),
\]
where the tropical addition $\tplus$ is the usual minimum and tropical
multiplication $\ttimes$ is the usual addition.  We denote by $\mathbb
B$ the Boolean subsemiring of $\Rbar$ consisting of $\{ 0 ,\infty \}$
with the induced operations.  We denote by
\[
S :=\Rbar[x_1^{\pm 1}, \dots,x_n^{\pm 1}] \quad \text{ and } \quad \tilde{S} :=\Rbar[x_0,\dots,x_n]
\]
the semirings of tropical Laurent polynomials and tropical polynomials
in the variables $\mathbf{x}=(x_1,\dots,x_n)$ and
$\mathbf{x}=(x_0,\dots,x_n)$, respectively. Elements of $S$ or
$\tilde{S}$ are (Laurent) polynomials with coefficients in $\Rbar$
where all operations are to be interpreted tropically. Explicitly, if
$F \in \tilde{S}$ then $F$ has the form $F(\mathbf{x}) = \bigtplus_{\mathbf{u} \in
  \mathbb N^{n+1}}a_\mathbf{u}\ttimes \mathbf{x}^{\mathbf{u}}$,
where $a_{\mathbf{u}} \in \Rbar$ and all but finitely many of the
$a_{\mathbf{u}}$ equal $\infty$.  As a function, this is
$F(\mathbf{x})=\min_{\mathbf{u} \in \mathbb N^{n+1}}(a_\mathbf{u} + \mathbf{x}
\cdot {\mathbf{u}})$.  Elements of $S$ have the form $F(\mathbf{x}) =
\bigtplus_{\mathbf{u} \in \mathbb Z^n} a_{\mathbf{u}} \ttimes
\mathbf{x}^{\mathbf{u}}$, where again $a_{\mathbf{u}} \in \Rbar$ and
all but finitely many $a_{\mathbf{u}}$ equal $\infty$.  Note that
elements of $S$ and $\tilde{S}$ are regarded as tropical polynomials,
not functions.  By this we mean that $F(x)=x^2 \tplus 0$ and $G(x)=x^2
\tplus 1 \ttimes x \tplus 0$ are different as elements of $S$, even
though $F(w)= \min(2w,0) = \min(2w, w+1, 0) =G(w)$ for all $w \in
\mathbb R$.

We adopt the notational convention that lower case letters denote
elements of the conventional (Laurent) polynomial ring with
coefficients in $K$ and upper case letters denote tropical (Laurent)
polynomials with coefficients in $\Rbar$. 

The {\em support} of a (Laurent)
polynomial $f = \sum c_{\mathbf{u}} \mathbf{x}^{\mathbf{u}}$ is 
the subset of $\mathbb N^{n+1}$ (respectively $\mathbb Z^n$) defined by 
$\supp(f) := \{\mathbf{u} : c_{\mathbf{u}} \neq 0\}$.
Similarly, for a tropical (Laurent) polynomial $F =
\bigtplus a_{\mathbf{u}}\ttimes \mathbf{x}^{\mathbf{u}}$ we write 
$\supp(F) := \{\mathbf{u} : a_{\mathbf{u}} \neq \infty\}.$
We call $a_{\mathbf{u}}$ the coefficient in $F$ of the monomial 
$\mathbf{u}$. 

Fix a field $K$ with a valuation $\val \colon K \rightarrow \Rbar$.
We write $R$ for the valuation ring $\{c \in K : \val(c) \geq 0 \}$,
and $\K$ for the residue field $R/\{c \in K : \val(c) >0\}$. 

 A polynomial $f \in K[x_1^{\pm 1},\dots,x_n^{\pm1}]$ gives rise to a
 tropical polynomial $\trop(f) \in S$ as follows.  If $f =
 \sum_{\mathbf{u} \in \mathbb Z^{n}} c_{\mathbf{u}}
 \mathbf{x}^{\mathbf{u}}$, then
\begin{align*}
\trop(f)& := \textstyle \bigtplus_{\mathbf{u} \in \mathbb Z^{n}} \val(c_{\mathbf{u}}) \ttimes \mathbf{x}^{\mathbf{u}}.
\end{align*}
The {\em tropical hypersurface} defined by $f$ is
\[
\trop(V(f)) := \{ \mathbf{w} \in \mathbb R^n : \text{the minimum in } \trop(f)(\mathbf{w}) \text{ is achieved at least twice} \}.
\]
The tropicalization of a variety $Y \subset (K^*)^n$ defined by an
ideal $I \subset K[x_1^{\pm 1},\dots,x_n^{\pm1}]$ is
\[
\trop(Y) := \bigcap_{f \in I} \trop(V(f)).
\] 
For more details on tropical varieties see~\cite{TropicalBook}.

Classically, a subscheme of the $n$-dimensional torus $T$ is defined by an ideal in
the Laurent polynomial ring $K[x_1^{\pm 1},\dots,x_n^{\pm 1}]$.  There
are two possible ways to tropicalize this.  The first gives an ideal
in the semiring $S$ of tropical Laurent polynomials.

\begin{definition}
Let  $I \subset K[x_1^{\pm 1},\dots,x_n^{\pm 1}]$  be an ideal.  The  ideal $\trop(I)$ in the semiring
$S$ is generated by the tropical polynomials
$\trop(f)$ for $f \in I$:
\[\trop(I):= \langle \trop(f) : f \in I \rangle.\]
The definition of $\trop(I)$ is the same for an ideal in
$K[x_0,\dots,x_n]$.
\end{definition}
If the value group $\Gamma := \im \val$ of $K$ equals
all of $\Rbar$ and the residue field $\K$ is infinite then every tropical
polynomial in the ideal $\trop(I)$ has the form $\trop(f)$ for some $f \in I$.
Indeed, in that case $a \ttimes  \mathbf{x}^{\mathbf{u}} \ttimes  \trop(f)
= \trop(c\,\mathbf{x}^{\mathbf{u}}f)$ for any $c \in K$ with $\val(c) = a$, and
$\trop(f) \tplus \trop(g) = \trop(f + \alpha g)$ for a sufficiently
general $\alpha \in K$ with $\val(\alpha) = 0$. 

A different approach to tropicalizing the scheme defined by $I$ is
given in~\cite{Giansiracusa2}.  Here the ideal $I$ gives rise to a
{\em congruence} on $S$.  This is an equivalence relation on $S$ that
is closed under tropical addition and tropical multiplication, i.e., 
if $F_1 \sim G_1$ and $F_2 \sim G_2$ then 
$F_1 \tplus F_2 \sim G_1 \tplus G_2$ and $ F_1 \ttimes F_2 \sim G_1 \ttimes G_2$.
If $\phi \colon S \rightarrow R$ is a semiring homomorphism, then $\{
F \sim G : \phi(F) = \phi(G) \}$ is a congruence, and all congruences
on $S$ arise in this fashion.  This is a key reason to consider
congruences instead of only ideals.
For a subset $\{ (F_{\alpha}, G_{\alpha}) : \alpha \in A\}$ of $S
\times S$ there is a smallest congruence on $S$ containing $F_{\alpha}
\sim G_{\alpha}$ for all $\alpha \in A$, which we denote by $\langle
F_{\alpha} \sim G_{\alpha} \rangle_{\alpha \in A}$.  
All these notions also 
make sense for the semiring $\tilde{S}$.

The following definitions are taken from Definitions 5.1.1 and 6.1.3 of
\cite{Giansiracusa2}.

\begin{definition} \label{d:bendrelations}
Let $F$ be a tropical (Laurent) polynomial.  For $\mathbf{v} \in
\supp(F)$ we write $F_{\hat{\mathbf{v}}}$ for the tropical polynomial
obtained by removing the term involving $\mathbf{v}$ from $F$.
Explicitly, if $F = \bigtplus a_{\mathbf{u}} \ttimes \mathbf{x}^{\mathbf{u}}$, then 
\[\textstyle F_{\hat{\mathbf{v}}} := \bigtplus_{\mathbf{u} \neq
  \mathbf{v}}a_{\mathbf{u}} \ttimes  \mathbf{x}^{\mathbf{u}}.
\] 
The {\em bend relations} of $F$ are:
\[B(F) :=\{ F \sim F_{\hat{\mathbf{v}}} : \mathbf{v} \in \supp(F) \}.\]
Given an ideal $I \subset K[x_1^{\pm 1},\dots,x_n^{\pm 1}]$, the
{\em scheme-theoretic tropicalization} of $I$ is the congruence on $S$
\[\Trop(I) := \langle B(\trop(f)) : f \in I \rangle.\]
We use the same definition for $\Trop(I)$ if $I$ is a homogeneous ideal in $K[x_0,\dots,x_n]$.
\end{definition}

In~\cite{Giansiracusa2} the authors show that the tropical variety of
an ideal $I$ can be recovered from the congruence $\Trop(I)$ as 
\begin{equation*}\label{eq:dual}
\trop(V(I)) =
\Hom(S/\Trop(I),\mathbb{R}),
\end{equation*}
where the homomorphisms are semiring
homomorphisms.  Explicitly, this means that 
\begin{equation} \label{eq:dual2}
\trop(V(I)) = \{
\mathbf{w} \in \mathbb R^n : \trop(f)(\mathbf{w}) =
\trop(f)_{\hat{\mathbf{v}}}(\mathbf{w}) \text{ for all } f \in I, \mathbf{v}
\in \supp(f) \}.
\end{equation}

\begin{remark}
When $I$ is a binomial ideal, an equivalent congruence appears in the
work of Kahle and Miller \cite{KahleMiller}.  For a binomial ideal $I
\subset K[x_1,\dots,x_n]$ they define a congruence on the monoid
$\mathbb N^n$ generated by the relations $ \{ \mathbf{u} \sim
\mathbf{v} : \exists \lambda \in K^* \text{ such that }
\mathbf{x}^{\mathbf{u}} - \lambda \mathbf{x}^{\mathbf{v}} \in I \}$.
These relations, together with $\mathbf{u} \sim \infty$
whenever $\mathbf{x}^{\mathbf{u}} \in I$, generate the congruence
$\Trop(I)$ on $\mathbb B[x_1,\dots,x_n]$ when $K$ has the trivial
valuation. 
\end{remark}

In the rest of this section we develop some basic properties of these
congruences, leading to a proof of part of Theorem~\ref{t:equivalent}.
We will make repeated use of the following result on congruences on
$S$.  By a monomial in $S$ we mean a tropical polynomial whose support
has size one.

\begin{lemma}\label{lem:semimodule}
The congruence $\langle F_\alpha \sim G_\alpha \rangle_{\alpha \in A}$ 
on $S$ is equal to the transitive closure of the set $U$ of relations of the form
\[
 M \ttimes F_\alpha \tplus H  \sim M \ttimes G_\alpha \tplus H 
\]
and their reverse, where $\alpha \in A$, $H \in S$, and $M$ is a monomial in $S$.

Thus the congruence $\Trop(I)$ is equal to the transitive closure of
the set of relations of the form 
\begin{align*}
 & a \ttimes \trop(f)_{\hat{\mathbf{v}}} \tplus H \sim a \ttimes \trop(f) \tplus H
\end{align*}
and their reverse,
where $f \in I$, $\mathbf{v} \in \supp(f)$,
$a \in \mathbb R$, and $H \in S$.  
\end{lemma}
\begin{proof}

By \cite{Giansiracusa2}*{Lemma 2.4.5} we know that $\langle F_\alpha
\sim G_\alpha \rangle_{\alpha \in A}$ is the transitive closure of the
subsemiring of $S \times S$ generated by the elements $F_\alpha \sim
G_\alpha$, $G_\alpha \sim F_\alpha$, and $1 \sim 1$.  We first show
that this is in fact the transitive closure $T$ of the
$S$-subsemimodule $N$ (as opposed to the $S$-subsemiring) of $S \times
S$ generated by these elements.  Let $F \sim G$ and $F' \sim G'$ be
elements of $T$. We will show that their tropical product and sum
is also in $T$, and thus $T$ is a subsemiring of $S \times S$, as desired.  By
definition, there exist chains $F \sim H_1 \sim \dotsb \sim H_l
\sim G$ and $F' \sim H'_1 \sim \dotsb \sim H'_{l'} \sim G'$ 
of relations in $N$.  The fact that $F \ttimes F' \sim G \ttimes G'$ is
in $T$ follows from the chain of relations in $N$
 \[
  F \ttimes F' \sim H_1 \ttimes F' \sim  \dotsb \sim H_{l} \ttimes F' \sim G \ttimes F'
  \sim G \ttimes H'_1 \sim \dotsb \sim G \ttimes H'_{l'} \sim G \ttimes G'.
 \]
A similar argument shows that the tropical sum $F \tplus F' \sim G \tplus G'$ is in $T$.

We now prove that all relations in $N$ are in the transitive
closure of the set $U$. Any relation in $N$ has the form
 \begin{equation}\label{eq:relation}
\textstyle (\bigtplus_{i=1}^s Q_i \ttimes C_i) \tplus  Q  \sim  (\bigtplus_{i=1}^s Q_i \ttimes D_i) \tplus Q ,
 \end{equation}
 where all the $Q_i$ are in $S$, all the relations $C_i \sim D_i$ have the form
 $F_\alpha \sim G_\alpha$ or their reverse, and $Q \in S$.  By
 allowing some of the relations $C_i \sim D_i$ to be equal, we can
 assume that the $Q_i$ are monomials in $S$.  For $l = 0, 1,
 \dotsc, s$, let $H_l \in S$ be defined by
 \[
  \textstyle H_l :=(\bigtplus_{i=1}^l  Q_i \ttimes D_i) \tplus ( \bigtplus_{i=l+1}^s Q_{i} \ttimes C_{i}) \tplus Q.
 \]
 Note that $H_0 \sim H_1 \sim \dotsb \sim H_s$ is a chain
 of relations in $U$. The relation \eqref{eq:relation} is simply $H_0
 \sim H_s$, so it is in the transitive closure of $U$.

The last claim of the lemma follows from the fact that $\Trop(I)$ is generated by
the relations $\trop(f) \sim \trop(f)_{\hat{\mathbf{v}}}$ for $f \in
I$.  If $f \in I$ then $\mathbf{x}^{\mathbf{u}}f \in I$, so we may replace the
tropical monomial $M$ by a scalar $a$.  
\end{proof}

\begin{remark}
If the value group $\Gamma = \im \val$ equals all of $\Rbar$ then for
all scalars $a \in \mathbb R$ we can find $c \in K$ with
$\val(c) = a$, so $a \ttimes \trop(f) = \trop(c f)$. Therefore, in
this case the congruence $\Trop(I)$ can be described as the transitive
closure of the set of relations of the form
$\trop(f)_{\hat{\mathbf{v}}} \tplus H \sim \trop(f) \tplus H$ and
their reverse, where $f \in I$, $\mathbf{v} \in \supp(f)$, and $H \in
S$.
\end{remark}

The following proposition is the key technical result that is needed to 
prove Theorem~\ref{t:mainthm} and parts of Theorem~\ref{t:equivalent}. 

\begin{proposition} \label{p:chaininsupport}
Let $I$ be an ideal in $K[x_1^{\pm 1},\dots,x_n^{\pm 1}]$, and let $F
\sim G$ be a relation in the congruence $\Trop(I)$ on $S$, where $F =
\bigtplus \alpha_{\mathbf{u}} \ttimes \mathbf{x}^{\mathbf{u}}$ and $G=
\bigtplus \beta_{\mathbf{u}} \ttimes \mathbf{x}^{\mathbf{u}}$.  Then there
is a chain $F =F_0\sim F_1 \sim \dots \sim F_s \sim F_{s+1} =G$ of
relations in $\Trop(I)$ satisfying the following two properties.
\begin{enumerate}[label={\normalfont (\alph*)}]
 \item Each $F_i \sim F_{i+1}$ has
the form $m \ttimes \trop(g) \tplus H \sim
m \ttimes \trop(g)_{\hat{\mathbf{v}}} \tplus H$ or the reverse, for some $g \in
I$, $\mathbf{v} \in \supp(g)$, $H \in S$, and $m \in \mathbb R$.
 \item The coefficient $\gamma_{\mathbf{u},i}$ of $\mathbf{u}$ in $F_i$ equals either $\alpha_{\mathbf{u}}$ or $\beta_{\mathbf{u}}$. 
\end{enumerate}
\end{proposition}

\begin{proof}
By Lemma~\ref{lem:semimodule} there is a chain $F=F_0
\sim F_1 \sim \dots \sim F_s \sim F_{s+1}=G$ of relations in $\Trop(I)$ 
with the property that
for each $i$ we have $F_i \sim F_{i+1}$ equal to
$m_i \ttimes \trop(g_i) \tplus H_i \sim
m_i \ttimes \trop(g_i)_{\hat{\mathbf{v}}} \tplus H_i$ or the reverse, for some
polynomial $g_i \in I$, $\mathbf{v} \in \supp(g_i)$, $H_i \in S$, and
$m_i \in \mathbb R$.
We now show that we can modify this chain to 
get a chain where the coefficients have the required form.
We represent the given chain by a path of length $s+1$ with vertices
labeled by the $F_i$ and an oriented edge labeled by $\mathbf{v}$ from
$m_i \ttimes \trop(g_i) \tplus H_i$ to $m_i \ttimes \trop(g_i)_{\hat{\mathbf{v}}} \tplus H_i$.  

We claim that we can locally modify the path by switching the
order of adjacent edges or amalgamating edges if the labels agree,
in the following six ways:
\begin{enumerate}
\item $F_{i-1} \stackrel{\mathbf{v}}{\rightarrow} F_{i}
  \stackrel{\mathbf{v}'}{\leftarrow} F_{i+1}$ can be replaced by 
$F_{i-1} \stackrel{\mathbf{v}'}{\leftarrow} F'_{i}
  \stackrel{\mathbf{v}}{\rightarrow} F_{i+1}$,

\item $F_{i-1} \stackrel{\mathbf{v}}{\rightarrow} F_{i}
  \stackrel{\mathbf{v}'}{\rightarrow} F_{i+1}$ can be replaced by 
 $F_{i-1} \stackrel{\mathbf{v}'}{\rightarrow} F'_{i}
  \stackrel{\mathbf{v}}{\rightarrow} F_{i+1}$,
\item $F_{i-1} \stackrel{\mathbf{v}}{\leftarrow} F_{i}
\stackrel{\mathbf{v}'}{\leftarrow} F_{i+1}$ can be replaced by 
$F_{i-1}
    \stackrel{\mathbf{v}'}{\leftarrow} F_{i+1}' \stackrel{\mathbf{v}}{\leftarrow}
    F_{i+1}$,
\item $F_{i-1} \stackrel{\mathbf{v}}{\leftarrow} F_{i}
  \stackrel{\mathbf{v}}{\rightarrow} F_{i+1}$ can be replaced by one
  of $F_{i-1} \stackrel{\mathbf{v}}{\leftarrow} F_{i+1}$, $F_{i-1}
  \stackrel{\mathbf{v}}{\rightarrow} F_{i+1}$, or $F_{i-1} = F_{i+1}$,

\item $F_{i-1} \stackrel{\mathbf{v}}{\rightarrow} F_i
  \stackrel{\mathbf{v}}{\rightarrow} F_{i+1}$ 
can be replaced by $F_{i-1}\stackrel{\mathbf{v}}{\rightarrow} F_{i+1}$, and 

\item $F_{i-1} \stackrel{\mathbf{v}}{\leftarrow} F_i
  \stackrel{\mathbf{v}}{\leftarrow} F_{i+1}$ 
can be replaced by $F_{i-1}\stackrel{\mathbf{v}}{\leftarrow} F_{i+1}$.
\end{enumerate}

By repeated use of the first of these operations we may assume that
all left-pointing arrows in the path come before all right-pointing arrows.  If
$\mathbf{v}$ appears as an arrow label on more than one arrow, by
repeated use of the second and third operations we may assume that the
left-pointing arrows labeled by $\mathbf{v}$ are the last
left-pointing arrows, and the right-pointing arrows labeled by
$\mathbf{v}$ are the first right-pointing arrows.  By repeated use of
the last three operations we can then replace these arrows by at most
one arrow labeled by $\mathbf{v}$.  In this fashion we get a new chain
$F=F_0 \sim F_1 \sim \dots \sim F_s \sim F_{s+1}=G$ where each arrow
label occurs exactly once.  As the coefficient of $\mathbf{v}$ in
$F_i$ equals that in $F_{i+1}$ unless the arrow between $F_i$ and
$F_{i+1}$ is labeled by $\mathbf{v}$, this means that the coefficient
of $\mathbf{v}$ changes at most once in the path from $F$ to $G$, so for
all $F_i$ the coefficient of $\mathbf{v}$ equals the coefficient of
$\mathbf{v}$ in either $F$ or $G$.

It thus suffices to show that these six arrow replacements can be
made.  In each case we make use of the fact that the coefficients of
$F_{i-1}, F_i$ and $F_{i+1}$ all agree for monomials $\mathbf{u} \neq
\mathbf{v}, \mathbf{v}'$.

$\bullet$ {\bf Case} $F_{i-1} \stackrel{\mathbf{v}}{\rightarrow} F_{i}
\stackrel{\mathbf{v}'}{\leftarrow} F_{i+1}$.  By assumption we have $F_{i-1} =
m_{i-1} \ttimes \trop(g_{i-1}) \tplus H_{i-1}$, $F_i = m_{i-1} \ttimes
\trop(g_{i-1})_{\hat{\mathbf{v}}} \tplus H_{i-1} = m_{i+1} \ttimes
\trop(g_{i+1})_{\hat{\mathbf{v}}'} \tplus H_{i+1}$, and $F_{i+1} =
m_{i+1} \ttimes \trop(g_{i+1}) \tplus H_{i+1}$.
Let $c_{\mathbf{v}}$ be the coefficient of
$\mathbf{x}^{\mathbf{v}}$ in $g_{i-1}$, and let $d_{\mathbf{v}'}$ be
the coefficient of $\mathbf{x}^{\mathbf{v}'}$ in $g_{i+1}$.  Let
$H'_{i-1} = H_{i-1} \tplus m_{i+1} \ttimes \val(d_{\mathbf{v}'})
\ttimes \mathbf{x}^{\mathbf{v}'}$, and $ H'_{i+1} = H_{i+1} \tplus
m_{i-1} \ttimes \val(c_{\mathbf{v}}) \ttimes \mathbf{x}^{\mathbf{v}}$.
Set $F'_i = m_{i-1} \ttimes \trop(g_{i-1}) \tplus H'_{i-1}$, and note
that this equals $m_{i+1} \ttimes \trop(g_{i+1})\tplus H'_{i+1}$.  In
addition, $F_{i-1} =
m_{i+1} \ttimes \trop(g_{i+1})_{\hat{\mathbf{v}}'} \tplus H'_{i+1}$ and
$F_{i+1} = m_{i-1} \ttimes \trop(g_{i-1})_{\hat{\mathbf{v}}} \tplus H'_{i-1}$.
We then have the relationship $F_{i-1}
\stackrel{\mathbf{v}'}{\leftarrow} F'_{i}
\stackrel{\mathbf{v}}{\rightarrow} F_{i+1}$ as required.

$\bullet$ {\bf Case} $F_{i-1} \stackrel{\mathbf{v}}{\rightarrow} F_{i}
  \stackrel{\mathbf{v}'}{\rightarrow} F_{i+1}$.  By assumption 
 $F_{i-1}=m_{i-1} \ttimes \trop(g_{i-1}) \tplus H_{i-1}$, $F_i =
  m_{i-1}\ttimes \trop(g_{i-1})_{\hat{\mathbf{v}}} \tplus H_{i-1} =
  m_i \ttimes \trop(g_i) \tplus H_i$, and $F_{i+1} =
  m_i \ttimes \trop(g_i)_{\hat{\mathbf{v}}'} \tplus H_i$.
	Write $g_{i-1} = \sum c_{\mathbf{u}} \mathbf{x}^{\mathbf{u}}$, and $g_i = \sum
  d_{\mathbf{u}} \mathbf{x}^{\mathbf{u}}$.

 Set $H'_{i} = H_i \tplus  m_{i-1} \ttimes \val(c_{\mathbf{v}}) \ttimes \mathbf{x}^{\mathbf{v}}$.  We have $F_{i-1} =
m_i \ttimes \trop(g_i) \tplus H'_{i}$.  Set
 $F_i'=m_i \ttimes \trop(g_i)_{\hat{\mathbf{v}}'} \tplus H'_{i}$.  We now
 have two further subcases.  Let $b_{\mathbf{v}'}$ be the coefficient
 of $\mathbf{v}'$ in $H_i$.
\begin{enumerate}
\item $b_{\mathbf{v}'} \leq m_{i-1}+\val(c_{\mathbf{v}'}) $.  Set
  $H'_{i-1}= (H_{i-1})_{\hat{\mathbf{v}}'} \tplus 
  b_{\mathbf{v}'} \ttimes \mathbf{x}^{\mathbf{v}'}$.  In this case $F_i'$
  equals $m_{i-1}\ttimes \trop(g_{i-1}) \tplus  H'_{i-1}$.  We then have
  $F_{i+1} =
  m_{i-1} \ttimes \trop(g_{i-1})_{\hat{\mathbf{v}}} \tplus H'_{i-1}$, and thus
  $F_{i-1} \stackrel{\mathbf{v}'}{\rightarrow} F'_i
  \stackrel{\mathbf{v}}{\rightarrow} F_{i+1}$ as required.

\item $b_{\mathbf{v}'}> m_{i-1}+\val(c_{\mathbf{v}'})$.  This implies
  in particular that $c_{\mathbf{v}'} \neq 0$.  Let $h = g_{i-1} -
  (c_{\mathbf{v}'}/d_{\mathbf{v}'}) g_i = \sum (c_{\mathbf{u}} -
  d_{\mathbf{u}} (c_{\mathbf{v}'}/d_{\mathbf{v}'})) \mathbf{x}^{\mathbf{u}}$.  By
  construction $\mathbf{v}' \not \in \supp(h)$.  Set $H'_{i+1} = F_{i+1}$.

  We claim that $F'_i = m_{i-1} \ttimes \trop(h) \tplus  H'_{i+1}$.  
The coefficient of $\mathbf{v}'$ in $F_i$ is 
$m_i + \val(d_{\mathbf{v}'})$, since $F_i \neq F_{i+1}$, so comparing
the two different expressions for $F_i$ we see that
$m_i + \val(d_{\mathbf{v}'}) \leq m_{i-1}+\val(c_{\mathbf{v}'})$.  Thus
$\val(c_{\mathbf{v}'}/d_{\mathbf{v}'}) \geq m_i-m_{i-1}$.  The
coefficient of $\mathbf{u}$ in $m_{i-1} \ttimes \trop(h)$ is $m_{i-1} +
\val(c_{\mathbf{u}}-
d_{\mathbf{u}}(c_{\mathbf{v}'}/d_{\mathbf{v}'}))$, which is at least
$m_{i-1} + \min(\val(c_{\mathbf{u}}), \val(d_{\mathbf{u}}) +
\val(c_{\mathbf{v}'}/d_{\mathbf{v}'}))$.  This in turn is at
least $\min(m_{i-1}+\val(c_{\mathbf{u}}), m_i +
\val(d_{\mathbf{u}}))$.  For $\mathbf{u} \neq \mathbf{v}, \mathbf{v}'$
both terms in this minimum are at least the coefficient of
$\mathbf{u}$ in $F_{i-1}$, which equals that in $F_{i+1}=H'_{i+1}$.  For
$\mathbf{u}=\mathbf{v}$, $m_{i-1} + \val(c_{\mathbf{v}}) <
m_i+\val(d_{\mathbf{v}})$, since $F_{i-1} \neq F_i$, so
$\val(c_{\mathbf{v}}-d_{\mathbf{v}}(c_{\mathbf{v}'}/d_{\mathbf{v}'}))
= \val(c_{\mathbf{v}})$. The coefficient of $\mathbf{v}$ in
$m_{i-1}  \ttimes  \trop(h) \tplus H'_{i+1}$ is then equal to $m_{i-1}+\val(c_{\mathbf{v}})$,
which equals the coefficient in $F'_i$.  Finally, the coefficient
of $\mathbf{v}'$ is $b_{\mathbf{v}'}$, which is also the coefficient
of $\mathbf{v}'$ in $F'_i$.

Since the coefficient of $\mathbf{v}$ is $m_{i-1}+\val(c_{\mathbf{v}})
<\infty$, we have $\mathbf{v} \in \supp(\trop(h))$, and thus $F_{i+1}
= m_{i-1} \ttimes  \trop(h)_{\hat{\mathbf{v}}} \tplus  H'_{i+1}$.
 This again gives the relation $F_{i-1}
 \stackrel{\mathbf{v}'}{\rightarrow} F'_i
 \stackrel{\mathbf{v}}{\rightarrow} F_{i+1}$.
\end{enumerate}

$\bullet$ {\bf Case} $F_{i-1} \stackrel{\mathbf{v}}{\leftarrow} F_{i}
  \stackrel{\mathbf{v}'}{\leftarrow} F_{i+1}$.  This is identical to the previous
  case with the roles of $F_{i-1}$ and $F_{i+1}$ reversed.  

$\bullet$ {\bf Case} $F_{i-1} \stackrel{\mathbf{v}}{\leftarrow} F_{i}
  \stackrel{\mathbf{v}}{\rightarrow} F_{i+1}$.  By assumption we have  $F_i = 
  m_i \ttimes \trop(g_i)\tplus H_i = m'_i \ttimes \trop(g'_i) \tplus H'_i$, $F_{i-1} =
  m_i \ttimes \trop(g_i)_{\hat{\mathbf{v}}} \tplus H_i$, and
  $F_{i+1}= m'_i\ttimes \trop(g_{i+1})_{\hat{\mathbf{v}}} \tplus H'_i$.
  Let $\gamma_{\mathbf{v},j}$ be the coefficient of $\mathbf{v}$ in $F_j$, for
  $j=i-1,i,i+1$.  
  We have $\gamma_{\mathbf{v},i-1}, \gamma_{\mathbf{v},i+1} >
  \gamma_{\mathbf{v},i}$.  If $\gamma_{\mathbf{v},i-1} = \gamma_{\mathbf{v},i+1}$ then
  $F_{i-1}=F_{i+1}$, so we may replace $F_{i-1}
  \stackrel{\mathbf{v}}{\leftarrow} F_{i} \stackrel{\mathbf{v}}{\rightarrow} F_{i+1}$ by
  just $F_{i-1}=F_{i+1}$.  If $\gamma_{\mathbf{v},i-1} < \gamma_{\mathbf{v},i+1}$, 
  let $\tilde{H}_i = F_{i+1}$.  Then $F_{i-1} =
  (\gamma_{\mathbf{v},i-1}-\gamma_{\mathbf{v},i}) \ttimes m_i' \ttimes \trop(g'_i) \tplus \tilde{H}_i$,
  and $F_{i+1} = (\gamma_{\mathbf{v},i-1}-\gamma_{\mathbf{v},i}) \ttimes m_i' \ttimes \trop(g'_i)_{\hat{\mathbf{v}}} \tplus 
  \tilde{H}_i$, so we can replace $F_{i-1} \stackrel{\mathbf{v}}{\leftarrow}
  F_{i} \stackrel{\mathbf{v}}{\rightarrow} F_{i+1}$ by $F_{i-1}
  \stackrel{\mathbf{v}}{\rightarrow} F_{i+1}$.  If $\gamma_{\mathbf{v},i-1} >
  \gamma_{\mathbf{v},i+1}$ then with the same construction we can replace
  $F_{i-1} \stackrel{\mathbf{v}}{\leftarrow} F_{i} \stackrel{\mathbf{v}}{\rightarrow}
  F_{i+1}$ by $F_{i-1} \stackrel{\mathbf{v}}{\leftarrow} F_{i+1}$.

$\bullet$ {\bf Case} $F_{i-1} \stackrel{\mathbf{v}}{\rightarrow} F_i
  \stackrel{\mathbf{v}}{\rightarrow} F_{i+1}$.  By assumption
  $F_{i-1}= m_{i-1} \ttimes \trop(g_{i-1}) \tplus H_{i-1}$, $F_i =
  m_{i-1}\ttimes \trop(g_{i-1})_{\hat{\mathbf{v}}} \tplus H_{i-1} =
  m_i \ttimes \trop(g_i)\tplus H_i$, and $F_{i+1} = m_i \ttimes
  \trop(g_i)_{\hat{\mathbf{v}}} \tplus H_i$.  Set $H'_i= F_{i+1}$.
  Then $F_{i-1} = m_{i-1} \ttimes \trop(g_{i-1}) \tplus H'_i$, and
  $F_{i+1} = m_{i-1} \ttimes \trop(g_{i-1})_{\hat{\mathbf{v}}} \tplus 
  H'_i$, so we may replace $F_{i-1}
  \stackrel{\mathbf{v}}{\rightarrow} F_i
  \stackrel{\mathbf{v}}{\rightarrow} F_{i+1}$ by $F_{i-1}
  \stackrel{\mathbf{v}}{\rightarrow} F_{i+1}$.
  
$\bullet$ {\bf Case} $F_{i-1} \stackrel{\mathbf{v}}{\leftarrow} F_i \stackrel{\mathbf{v}}{\leftarrow}
  F_{i+1}$.  This is identical to the previous case, with the roles of
  $F_{i-1}$ and $F_{i+1}$ reversed.
\end{proof}

The first of the six cases in the previous proof is the only one that cannot
be reversed.  Indeed, in the congruence $\Trop(\langle x+y \rangle)$
we have the relations $x \sim x \! \tplus \! y  \sim y$ but neither $x \sim
\infty$ nor $y \sim \infty$.

A congruence $J$ on $\tilde{S}$ or $S$ is {\em homogeneous} with
respect to a grading 
by $\deg(x_i) = \delta_i \in \mathbb Z$ if $J$ is generated by
relations of the form $F \sim G$ where $F$ and  $G$ are both homogeneous of
the same degree. For a tropical polynomial $F$ we write $F_d$ for its
homogeneous component of degree $d$, where $d \in \mathbb Z$.

\begin{proposition}  \label{l:troplineality}
Let $J$ be a homogeneous congruence on $\tilde{S}$ or $S$.  If $F \sim
G \in J$, then $F_d \sim G_d \in J$ for all $d \in \mathbb Z$.
\end{proposition}

\begin{proof}
Let $\mathcal J=\{ F_{\alpha} \sim G_{\alpha} : \alpha \in A \}$ be a
homogeneous generating set for $J$, and fix $F \sim G \in J$.  By
Lemma~\ref{lem:semimodule} there is a chain $F = F_0
\sim F_1 \sim \dots \sim F_s \sim F_{s+1} =G$ of relations in $J$ where $F_i \sim F_{i+1}$
has the form $M_i \ttimes H_i \tplus P_i \sim M_i \ttimes H'_{i} \tplus P_i$ with $H_i
\sim H'_{i} \in \mathcal J$ (or its reverse), $M_i$ a monomial in $S$, and $P_i \in
S$.  The $d$th graded piece of $F_i \sim F_{i+1}$ either has the form
$Q_i \sim Q_i$ or $M_i \ttimes H_i \tplus Q_i \sim M_i \ttimes H'_{i} \tplus Q_i$ where
$Q_i$ is homogeneous of degree $d$.  Thus $(F_i)_d \sim (F_{i+1})_d \in
J$, so $F_d \sim G_d \in J$.
\end{proof}

We will also need the notion of the homogenization of a congruence on
$S$.  This will play the same role as the homogenization of an ideal
in the usual Laurent polynomial ring with coefficients in $K$.
Geometrically, this is the tropical analogue of taking the projective
closure of a subvariety of $(K^*)^n$.

Recall that the homogenization of a polynomial $f = \sum c_\mathbf{u}
\mathbf{x}^\mathbf{u} \in K[x_1,\dots,x_n]$ is the polynomial $\tilde{f} :=
\sum c_\mathbf{u} \mathbf{x}^\mathbf{u} x_0^{\deg(f) - |\mathbf{u}|} \in K[x_0, \dots,x_n]$, where
$|\mathbf{u}| = u_1 + \dotsb + u_n$, and $\deg(f)$ is the maximum of
$|\mathbf{u}|$ for which $c_\mathbf{u} \neq 0$.  For an ideal $I \subset
K[x_1^{\pm 1},\dots,x_n^{\pm 1}]$, its homogenization is the ideal $I^h
= \langle \tilde{f} : f \in I \cap K[x_1,\dots,x_n] \rangle$.

\begin{definition}
Given $F = \bigtplus a_{\mathbf{u}} \ttimes  \mathbf{x}^{\mathbf{u}} \in S$,
we denote by $\deg(F)$ the maximum $\max(|\mathbf{u}| : a_{\mathbf{u}}
\neq \infty)$.  If $\supp(F) \subset \mathbb N^n$, we write
$\tilde{F}$ for the tropical polynomial in $\tilde{S}$ given by
\[\tilde{F} := \textstyle \bigtplus a_{\mathbf{u}} \ttimes \mathbf{x}^{\mathbf{u}} 
\ttimes x_0^{\deg(F)-|\mathbf{u}|}.\]
If $F \sim G$ is a relation,  where $F, G$ are tropical polynomials 
(with $\supp(F), \supp(G) \subset \mathbb N^n$) 
satisfying $\deg(F) \geq \deg(G)$,
its homogenization is
\[
\widetilde{F \sim G} \quad := \quad \tilde{F} \sim x_0^{\deg(F)-\deg(G)} \ttimes \tilde{G} .
\]
Let $J$ be a congruence on $S$.  The homogenization $J^h$ of
$J$ is the congruence 
\[
J^h := \langle \widetilde{ F \sim G} : F \sim G \in J \text{ and } \supp(F), \supp(G) \subset \mathbb N^n \rangle.
\]
\end{definition}

\begin{proposition} \label{p:homogenizecongruence}
Let $I$ be an ideal in $K[x_1^{\pm 1},\dots,x_n^{\pm 1}]$, and
let $I^h \subset K[x_0,\dots,x_n]$ be 
its homogenization.  
Then we have the equality of
congruences on $\tilde{S}$
\[\Trop(I^h) = \Trop(I)^h.\]
\end{proposition}

\begin{proof}
Let $f$ be a homogeneous polynomial in $I^h$.  Write $g = f|_{x_0=1}$.  Note that $g \in I$, and  $f = x_0^a \tilde{g}$ for some $a \geq 0$.  For a
monomial $\mathbf{x}^{\mathbf{u}} \in K[x_0,\dots,x_n]$ write $\mathbf{u}'$ for the projection of
$\mathbf{u}$ onto the last $n$ coordinates.

Choose $\mathbf{u} \in \supp(f)$, and consider the relation $\trop(f) \sim
\trop(f)_{\hat{\mathbf{u}}} \in \Trop(I^h)$.  The homogenization of the relation
$\trop(g) \sim \trop(g)_{\hat{\mathbf{u}}'} \in \Trop(I)$ is 
\begin{equation*} 
\label{eqtn:homogenization}
\widetilde{\trop(g)} \sim x_0^{\deg(g)-\deg(g_{\hat{\mathbf{u}}'})} \ttimes
  \widetilde{\trop(g)}_{\hat{\mathbf{u}}'}. 
\end{equation*} 
Multiplying this relation by $x_0^a$
gives the relation $\trop(f) \sim \trop(f)_{\hat{\mathbf{u}}}$.  Since these relations generate
$\Trop(I^h)$, it follows that $\Trop(I^h) \subset
\Trop(I)^h$.

For the converse, it suffices to consider a relation of the form
$\widetilde{F \sim G}$ for $F \sim G \in \Trop(I)$ with both $F, G$ 
non-Laurent tropical polynomials, and show that it is a relation in
$\Trop(I^h)$.  By Proposition~\ref{p:chaininsupport} we can find a
chain $F =F_0 \sim F_1 \sim \dots \sim F_s \sim F_{s+1}=G$ with $F_i
\sim F_{i+1} \in \Trop(I)$ of the form $a_i \ttimes \trop(h_i) \tplus H_i \sim
a_i \ttimes \trop(h_i)_{\hat{\mathbf{v}}} \tplus H_i$ for some $h_i \in I$,
$a_i \in \mathbb R$, and $H_i \in S$, and for which the coefficient of
$\mathbf{u}$ in $F_i$ equals the coefficient of $\mathbf{u}$ in either
$F$ or $G$.  This latter condition implies that if $\mathbf{u} \not
\in \supp(F) \cup \supp(G)$ then $\mathbf{u} \not \in \supp(F_i)$, so
in particular each $F_i$ 
has support in $\mathbb N^n$ 
and $\deg(F_i) \leq \max(\deg(F),\deg(G))$. 
The homogenization of the relation
$a_i \ttimes \trop(h_i) \tplus H_i  \sim
a_i \ttimes \trop(h_i)_{\hat{\mathbf{v}}} \tplus H_i$ equals 
\[a_i \ttimes  \trop(\tilde{h}_i) \ttimes  x_0^b \tplus  \tilde{H_i} \ttimes  x_0^d \sim 
a_i \ttimes \trop(\tilde{h}_i)_{\hat{\mathbf{v}}'} \ttimes  x_0^b \tplus  \tilde{H_i} \ttimes x_0^d,
\] 
where $\mathbf{v}' \in \mathbb N^{n+1}$ has last $n$ coordinates equal to
$\mathbf{v}$,
and the numbers $b,d$ satisfy
$b = \max(\deg(H_i) - \deg(h_i), 0)$ and
$d = \max(\deg(h_i) - \deg(H_i), 0)$. 
Since $\tilde{h}_i \in I^h$, we have $\trop(\tilde{h}_i)
\sim \trop(\tilde{h}_i)_{\mathbf{v}'} \in \Trop(I^h)$, and so
$\widetilde{F_i \sim F_{i+1}} \in \Trop(I^h)$.

Each relation $\widetilde{F_i \sim F_{i+1}}$ is homogeneous of degree
at most $\max(\deg(F),\deg(G))$. The right-hand side of
$\widetilde{F_{i-1} \sim F_i}$ and the left-hand side of
$\widetilde{F_{i} \sim F_{i+1}}$ are either identical or differ by a
factor of $x_0^b$, with $b \in \mathbb N$ equal to the difference
between their degrees.  Thus we can multiply both
sides of the lower degree relation by $x_0^b$ to get two relations whose adjacent
terms coincide.  Doing this for the string $\widetilde{F_0 \sim F_1},
\dots, \widetilde{F_s \sim F_{s+1}}$ gives a chain of relations in
$\Trop(I^h)$ of the same degree, whose first entry is $x_0^a \ttimes 
\widetilde{F}$ and whose last entry is $x_0^b \ttimes \widetilde{G}$ for some
$a, b \in \mathbb N$ with at most one of $a$ and $b$ nonzero.  Taking
the transitive closure we get $x_0^a \ttimes \widetilde{F} \sim x_0^b \ttimes
\widetilde{G} \in \Trop(I^h)$.  This relation equals $\widetilde{F
  \sim G}$, which completes the proof.
\end{proof}

Note that the use of Proposition~\ref{p:chaininsupport}
was key in the proof of Proposition
\ref{p:homogenizecongruence}.

We are now in position to prove the equivalence
(\ref{item:congruence}) $\Leftrightarrow$ (\ref{item:ideal}) of
Theorem~\ref{t:equivalent} from the introduction.

\begin{proof}[Proof of (\ref{item:congruence}) $\Leftrightarrow$ (\ref{item:ideal}) of Theorem~\ref{t:equivalent}]
We first show that the ideal $\trop(I)$ determines the congruence
$\Trop(I)$. 
Specifically, we will show that 
\[ \Trop(I) = \langle F \sim F_{\hat{\mathbf{u}}} : F \in \trop(I), \mathbf{u} \in \supp(F) \rangle.\]
The inclusion $\subseteq$ follows from the fact that $\Trop(I)$ is
generated by relations of the form $\trop(f) \sim
\trop(f)_{\hat{\mathbf{u}}}$ for $f \in I$, which have the form $F
\sim F_{\hat{\mathbf{u}}}$ for $F \in \trop(I)$.
To prove the reverse inclusion,  note that any $F \in \trop(I)$ has the form $\bigtplus_{1
  \leq i \leq s} a_i \ttimes \trop(f_i)$ for some $f_1,\dots,f_s \in I$
and $a_i \in \mathbb R$. If $\mathbf{u} \in \supp(F)$, the polynomial $F_{\hat{\mathbf{u}}}$ is
 $\bigtplus_{1 \leq i \leq s} a_i \ttimes \trop(f_i)_{\hat{\mathbf{u}}}$, where we set
$\trop(f_i)_{\hat{\mathbf{u}}} = \trop(f_i)$ if $\mathbf{u} \not \in
\supp(f_i)$.  Thus $F \sim F_{\hat{\mathbf{u}}}$ equals the 
tropical sum of
$a_i \ttimes \trop(f_i) \sim a_i \ttimes  \trop(f_i)_{\hat{\mathbf{u}}}$ for $1
\leq i \leq s$, and so it lies in $\Trop(I)$. 

To show that the congruence
$\Trop(I)$ determines the ideal $\trop(I)$, 
first note that by Proposition~\ref{p:homogenizecongruence} the
congruence $\Trop(I)$ determines the congruence $\Trop(I^h)$ on
$\tilde{S}$, where $I^h \subset K[x_0,\dots,x_n]$ is the
homogenization of the ideal $I$. 
For any $d \geq 0$, denote by $\trop(I^h)_d$ the degree $d$ part 
of the homogeneous ideal $\trop(I^h) \subset \tilde{S}$.
We may regard any homogeneous tropical polynomial $F$ of degree $d$ as a
tropical linear form $l_F$ on the space $\Rbar^{\binom{n+d}{d}}$, whose coordinates
are indexed by the monomials of degree $d$. Under this identification
the set $\trop(I^h)_d$ is a tropical linear space; in fact,
$\trop(I^h)_d$ is equal to the tropicalization $\trop(I^h_d)$
of the degree-$d$ part $I^h_d$ of $I^h$. Let $\ell_d$ be the
tropical linear space in $\Rbar^{\binom{n+d}{d}}$ on which the tropical linear forms
$l_F$ vanish for all $F \in \trop(I^h)_d$, i.e.,
$\ell_d$ is the set of $\mathbf{z} \in
\mathbb \Rbar^{\binom{n +d}{d}}$ where the minimum in 
$l_F(\mathbf{z})$ is achieved twice for all $F \in \trop(I^h)_d$.
We have that $\ell_d$ is the dual tropical linear space $(\trop(I^h)_d)^{\perp}$.
The collection of tropical linear forms $l_G$ that vanish on $\ell_d$
is $\ell_d^{\perp} = (\trop(I^h)_d)^{\perp\perp}$, which is equal to 
$\trop(I^h)_d$ \cite{RinconIsotropic}*{Corollary 6.15}.  
Note that if $G$ is a homogeneous tropical polynomial of degree $d$, 
the tropical linear form $l_G$ vanishes on $\ell_d$ if and only if
$l_G(\mathbf{z}) =  l_{G_{\hat{\mathbf{u}}}}(\mathbf{z})$ for all 
$\mathbf{z} \in \ell_d$ and $\mathbf{u} \in \supp(G)$. 
It follows that if $G \in \tilde{S}_d$ is such that its bend relations
$B(G)$ are in $\Trop(I^h)_d$, then $G \in \ell_d^{\perp} = \trop(I^h)_d$.
We therefore have
\[
 \trop(I^h)_d = \{ G \in \tilde{S}_d : B(G) \subset \Trop(I^h)_d \},
\]
which shows that $\Trop(I^h)$ determines $\trop(I^h)_d$ 
for all $d \geq 0$, and thus $\trop(I^h)$. Since
$\trop(I)$ is the ideal in $S$ generated by $\{ F|_{x_0=0} : F \in \trop(I^h) \}$,
it also determines $\trop(I)$.
\end{proof}

\begin{remark}
The tropical linear spaces $\ell_d$ also encode the valuated matroids of the vector spaces $I^h_d$, 
so we can see the third equivalence of Theorem~\ref{t:equivalent} 
from the previous argument as well. This is elaborated on in
Section~\ref{s:matroids}.
\end{remark}

\section{Multiplicities}
\label{s:multiplicity}
In this section we prove Theorem~\ref{t:mainthm}.  The strategy is to
define a Gr\"obner theory for congruences on the semiring of tropical
polynomials, which lets us determine the multiplicities from
the tropical scheme.

We first recall the definition of multiplicity for maximal cells of a
tropical variety.  For an irreducible $d$-dimensional subvariety $Y
\subset (K^*)^n$ the tropical variety $\trop(Y) \subset \mathbb R^n$
is the support of a pure $d$-dimensional $\Gamma$-rational polyhedral
complex.  This
means that $\trop(Y)$ is the union of a set $\Sigma$ of
$d$-dimensional polyhedra of the form $\{\mathbf{w} \in \mathbb R^n : A \mathbf{w}
\leq \mathbf{b} \}$ where $A \in \mathbb Q^{r \times n}$ and $\mathbf{b} \in \Gamma^r$
for some $r \in \mathbb{N}$, and these polyhedra
intersect only along faces.  See \cite{TropicalBook}*{Chapter 3} for
more details.

 Let $I \subset K[x_1^{\pm 1},\dots,x_n^{\pm 1}]$ be the ideal of $Y$.
 Fix a group homomorphism $\Gamma \rightarrow K^*$, which we write $w
 \mapsto t^w$, satisfying $\val(t^w)=w$.  This may require replacing
 $K$ by an extension field; see \cite{TropicalBook}*{Chapter 2}.  For
 $a$ in the valuation ring $R$ we write $\overline{a}$ for its image
 in the residue field $\K$.  Fix $\mathbf{w}$ in the relative interior of a
 $d$-dimensional polyhedron $\sigma \in \Sigma$.  We denote by
 $\inn_{\mathbf{w}}(I) \subset \K[x_1^{\pm 1},\dots,x_n^{\pm 1}]$ the
 initial ideal of $I$ with respect to $\mathbf{w}$, in the sense
 described in \cite{TropicalBook}*{Section 2.4}.  This is the ideal
 $\inn_{\mathbf{w}}(I) := \langle \inn_{\mathbf{w}}(f) : f \in I
 \rangle,$ where for $f = \sum c_{\mathbf{u}} \mathbf{x}^{\mathbf{u}}$
 the initial form $\inn_{\mathbf{w}}(f)$ equals
 $\sum_{\val(c_{\mathbf{u}}) + \mathbf{w} \cdot \mathbf{u} = \gamma}
 \overline{t^{-\val(c_{\mathbf{u}})} c_{\mathbf{u}}}\,
 \mathbf{x}^{\mathbf{u}}$, with $\gamma = \min ( \val(c_{\mathbf{u}})
 + \mathbf{w} \cdot \mathbf{u} ) = \trop(f)(\mathbf{w})$.

The
 multiplicity of $\mathbf{w}$ is the multiplicity of the initial ideal
 $\inn_{\mathbf{w}}(I)$:
\[\mult(\mathbf{w}) := \sum_{P} \mult(P,\inn_{\mathbf{w}}(I)),\] where
 the sum is over the minimal associated primes of
 $\inn_{\mathbf{w}}(I)$, and $\mult(P,\inn_{\mathbf{w}}(I))$ is the
 multiplicity of the associated primary component.
 See \cite{TropicalBook}*{Section 3.4} for more details.  If coordinates
 on the torus $(K^*)^n$ have been chosen so that
 $\inn_{\mathbf{w}}(I)$ has a generating set involving only the
 variables $x_{d+1},\dots,x_n$, then
 \[\mult(\mathbf{w}) = \dim_{\K}\bigl( \, \K[x_{d+1}^{\pm 1},\dots,x_n^{\pm
       1}]/(\inn_{\mathbf{w}}(I) \cap \K[x_{d+1}^{\pm
       1},\dots,x_n^{\pm 1}])\, \bigr)\] 
 (see \cite{TropicalBook}*{Lemma 3.44}).

We now extend the definition of initial ideals to congruences on
$\tilde{S}$ and $S$.  

\begin{definition} \label{d:initialcongruence}
Let $F = \bigtplus a_{\mathbf{u}} \ttimes  \mathbf{x}^{\mathbf{u}} \in \tilde{S}$ and $\mathbf{w} \in \mathbb R^{n+1}$.  
  The {\em initial form} of $F$ with respect to $\mathbf{w}$ is 
the tropical polynomial in $\mathbb B[x_0,\dots,x_n]$
\[
\inn_{\mathbf{w}}(F) := \bigtplus_{a_{\mathbf{u}} + \mathbf{w} \cdot \mathbf{u} =F(\mathbf{w})}
\mathbf{x}^{\mathbf{u}}.
\] 
For $G = \bigtplus b_{\mathbf{v}} \ttimes \mathbf{x}^{\mathbf{v}} \in \tilde{S}$, let $\gamma =
\min(F(\mathbf{w}),G(\mathbf{w}))$.  The initial form of the relation
$F \sim G$ with respect to $\mathbf{w}$ is the relation
\[
\inn_{\mathbf{w}}(F \sim G) \quad := \quad \bigtplus_{a_{\mathbf{u}} + \mathbf{w} \cdot \mathbf{u} =\gamma}
\mathbf{x}^{\mathbf{u}} \sim \bigtplus_{b_{\mathbf{v}}+\mathbf{w} \cdot \mathbf{v} =\gamma}
\mathbf{x}^{\mathbf{v}}.
\]
 Note that if
$F(\mathbf{w})=G(\mathbf{w})$ then this is $\inn_{\mathbf{w}}(F) \sim
  \inn_{\mathbf{w}}(G)$, but if $F(\mathbf{w}) < G(\mathbf{w})$ then
  this is $\inn_{\mathbf{w}}(F) \sim \infty$.

For a congruence $J$ on $\tilde{S}$, the {\em initial congruence} of
$J$ with respect to $\mathbf{w}$ is the congruence on $\mathbb B[x_0,\dots,x_n]$ 
\[\inn_{\mathbf{w}}(J) := \langle \inn_{\mathbf{w}}(F \sim G) : F \sim G \in J \rangle.\]

The initial form with respect to $\mathbf{w} \in \mathbb R^n$ of a
relation between tropical Laurent polynomials
 and the initial congruence of a
congruence on $S$ are defined analogously.  
\end{definition}

\begin{example}
Consider $S = \Rbar[x^{\pm 1}, y^{\pm 1}, z^{\pm 1}]$, and let $F =
0 \ttimes x \tplus 1\ttimes y \tplus 2 \ttimes z \in S$.  For $\mathbf{u} 
=(1,0,0)$ we have the relation $F \sim F_{\hat{\mathbf{u}}}$, which is
$0 \ttimes x \tplus 1 \ttimes y \tplus 2 \ttimes z \sim 1\ttimes y \tplus 2 \ttimes z$. If $\mathbf{w} = (2,1,3)$, the initial form
$\inn_{\mathbf{w}}(F \sim F_{\hat{\mathbf{u}}})$ of this relation is $x \tplus y
\sim y$.
For $\mathbf{w}=(1,2,2)$ the initial form is $x \sim \infty$.
\end{example}

As in standard Gr\"obner theory, the initial congruence of a congruence
generated by $\{F_{\alpha} \sim G_{\alpha}\}_{\alpha \in A}$ for some
set $A$ is not necessarily generated by
$\{\inn_{\mathbf{w}}(F_{\alpha} \sim G_{\alpha})\}_{\alpha \in A}$.
For example, for $\mathbf{w} = (0,1,2)$ and the congruence $J$ on
$\Rbar[x,y,z]$ generated by $\{ x \sim y, x \sim z \}$, we have $y
\sim z \in J$, so $y \sim \infty \in \inn_{\mathbf{w}}(J)$. However, the
initial form of both $x \sim y$ and $x \sim z$ is $x \sim \infty$, and
$y \sim \infty \not \in \langle x \sim \infty \rangle$.

Definition~\ref{d:initialcongruence} is designed to commute with
tropicalization of polynomials, as the following lemma shows.

\begin{lemma} \label{l:initialformtrop}
For  $f \in K[x_1^{\pm 1},\dots,x_n^{\pm 1}]$ and
$\mathbf{w} \in \mathbb R^{n}$ we have
\[\inn_{\mathbf{w}}(\trop(f)) = \trop(\inn_{\mathbf{w}}(f)).\]
The same holds for $f \in K[x_0,\dots,x_n]$ and $\mathbf{w} \in
\mathbb R^{n+1}$.
\end{lemma}

\begin{proof}
Suppose $f = \sum c_\mathbf{u} \mathbf{x}^{\mathbf{u}}$ with
$c_{\mathbf{u}} \in K$, so $\trop(f) = \bigtplus \val(c_{\mathbf{u}}) \ttimes
\mathbf{x}^{\mathbf{u}}$.  Let $\gamma = \trop(f)(\mathbf{w})$.
By definition, $\inn_{\mathbf{w}}(f) =
\sum_{\val(c_{\mathbf{u}})+\mathbf{w} \cdot \mathbf{u} = \gamma}
\overline{t^{-\val(c_{\mathbf{u}})}c_{\mathbf{u}}}\mathbf{x}^{\mathbf{u}}
\in \K[x_1^{\pm 1},\dots,x_n^{\pm 1}]$.  Thus $\trop(\inn_{\mathbf{w}}(f)) =
\bigtplus_{\val(c_{\mathbf{u}})+\mathbf{w} \cdot \mathbf{u} = \gamma} 
\mathbf{x}^{\mathbf{u}} = \inn_{\mathbf{w}}(\trop(f))$, as
claimed.
\end{proof}

The first key result of this section is the following, which says that taking
congruences commutes with taking initial ideals.

\begin{proposition} \label{p:inicommuteswithtrop}
Let $I$ be an ideal in $K[x_1^{\pm 1},\dots,x_n^{\pm 1}]$.  Then for
$\mathbf{w} \in \mathbb R^n$ we have 
\[\inn_{\mathbf{w}}(\Trop(I)) =
\Trop( \inn_{\mathbf{w}}(I)).\]
\end{proposition}

\begin{proof}
Fix $\mathbf{w} \in \mathbb R^n$.  The congruence
$\Trop(\inn_{\mathbf{w}}(I))$ is generated by relations of the form
$\trop(g) \sim \trop(g)_{\hat{\mathbf{v}}}$ for $g \in
\inn_{\mathbf{w}}(I)$ and $\mathbf{v} \in \supp(g)$.  We first note
that we can write $g = \sum \inn_{\mathbf{w}}(f_i)$ for some $f_i \in
I$ with $\supp(\inn_{\mathbf{w}}(f_i)) \cap
\supp(\inn_{\mathbf{w}}(f_j)) =\emptyset$ if $i \neq j$.  Indeed, if
$g = \sum a_i \mathbf{x}^{\mathbf{u}_i} \inn_{\mathbf{w}}(f_i)$ for
$a_i \in \K$ and $f_i \in I$, then for $c_i \in R$ with
$\overline{c}_i= a_i$ we have $g = \sum \inn_{\mathbf{w}}(c_i
\mathbf{x}^{\mathbf{u}_i}f_i)$, so we may assume that $\mathbf{u}_i =
\mathbf{0}$ and $a_i=1$.  If the minimum in both
$\trop(f_i)(\mathbf{w})$ and $\trop(f_j)(\mathbf{w})$ is achieved at
the term involving $\mathbf{u}$, where the coefficient of
$\mathbf{x}^{\mathbf{u}}$ in $f_i$ is $c$ and the coefficient in $f_j$
is $d$, then $\gamma :=\trop(f_j)(\mathbf{w}) - \trop(f_i)(\mathbf{w})
= \val(d) - \val(c) \in \Gamma$, and we can find $\alpha \in K$ with
$\val(\alpha) = \val(d)-\val(c)$ and $\overline{\alpha
  t^{-\val(\alpha)}} = 1$.  We then have $h= f_j+\alpha f_i \in I$,
and $\inn_{\mathbf{w}}(h) =
\inn_{\mathbf{w}}(f_i)+\inn_{\mathbf{w}}(f_j)$.  We may thus replace
$f_i, f_j$ by $h$, and repeat this procedure until the supports of the
$\inn_{\mathbf{w}}(f_i)$ are disjoint.  Note that this implies that
$\trop(g) = \bigtplus \trop(\inn_{\mathbf{w}}(f_i))$.  

Now, for $\mathbf{v} \in \supp(\inn_{\mathbf{w}}(f_1))$ we can write
$H = \bigtplus_{i=2}^s(\trop(\inn_{\mathbf{w}}(f_i)))$, so $\trop(g) \sim
\trop(g)_{\hat{\mathbf{v}}}$ is equal to
$\trop(\inn_{\mathbf{w}}(f_1)) \tplus H \sim
\trop(\inn_{\mathbf{w}}(f_1))_{\hat{\mathbf{v}}} \tplus H$. This shows
that $\Trop(\inn_{\mathbf{w}}(I))$ is generated by relations of the
form $\trop(\inn_{\mathbf{w}}(f)) \sim
\trop(\inn_{\mathbf{w}}(f))_{\hat{\mathbf{v}}}$.  Since
$\min(\trop(f)(\mathbf{w}), \trop(f)_{\hat{\mathbf{v}}}(\mathbf{w})) =
\trop(f)(\mathbf{w})$, we have that $\inn_{\mathbf{w}}(\trop(f) \sim
\trop(f)_{\hat{\mathbf{v}}})$ equals $\inn_{\mathbf{w}}(\trop(f)) \sim
\inn_{\mathbf{w}}(\trop(f))_{\hat{\mathbf{v}}}$, so by Lemma
\ref{l:initialformtrop}, $\inn_{\mathbf{w}}(\trop(f) \sim
\trop(f)_{\hat{\mathbf{v}}})$ is equal to $\trop(\inn_{\mathbf{w}}(f))
\sim \trop(\inn_{\mathbf{w}}(f))_{\hat{\mathbf{v}}}$.  Note that the
term $\trop(\inn_{\mathbf{w}}(f))_{\hat{\mathbf{v}}}$ may equal
$\infty$.  This proves the containment $\inn_{\mathbf{w}}(\Trop(I))
\supseteq \Trop( \inn_{\mathbf{w}}(I))$.

For the reverse inclusion, let $(F' \sim G') = \inn_{\mathbf{w}}( F
\sim G)$ be a generator of the congruence
$\inn_{\mathbf{w}}(\Trop(I))$, where $F \sim G \, \in \Trop(I)$.  Fix
a chain $F = F_0 \sim F_1\sim \dots \sim F_s \sim F_{s+1}=G$ in
$\Trop(I)$ with $F_i= m_i \ttimes \trop(g_i) \tplus H_i$ and
$F_{i+1}=m_{i} \ttimes \trop(g_i)_{\hat{\mathbf{v}}} \tplus H_i$ (or
the reverse), satisfying the conditions of
Proposition~\ref{p:chaininsupport}. In particular, we have $\gamma :=
\min( F(\mathbf{w}), G(\mathbf{w})) \leq F_i(\mathbf{w})$ for all $i$.
For any $F_i = \bigtplus a_\mathbf{u} \ttimes \mathbf{x}^{\mathbf{u}}$ in this
chain, define $F'_i := \bigtplus_{a_\mathbf{u} + \mathbf{x} \cdot
  {\mathbf{u}} = \gamma} \mathbf{x}^{\mathbf{u}}$. Note that $F'_i$
might be equal to $\infty$.  We claim that the chain $F' = F'_0 \sim
F'_1\sim \dots \sim F'_s \sim F'_{s+1}=G'$ is a chain of relations in
$\Trop(\inn_{\mathbf{w}}(I))$.  It follows that $F' \sim G' \, \in
\Trop(\inn_{\mathbf{w}}(I))$, completing the proof.

To prove the claim, consider first the case where $F_i(\mathbf{w}) =
F_{i+1}(\mathbf{w}) = \gamma$ for some $i$. If $m_i + \trop(g_i)(\mathbf{w}) > H_i(\mathbf{w}) = \gamma$ then $(F'_i \sim
F'_{i+1}) = (\inn_{\mathbf{w}}(H_i) \sim \inn_{\mathbf{w}}(H_i))$.  If
$m_i+\trop(g_i)(\mathbf{w}) = H_i(\mathbf{w}) = \gamma$ then $F'_i
\sim F'_{i+1}$ equals 
\[\inn_{\mathbf{w}}(\trop(g_i)) \tplus  \inn_{\mathbf{w}}(H_i) \sim
\inn_{\mathbf{w}}(\trop(g_i))_{\hat{\mathbf{u}}} \tplus \inn_{\mathbf{w}}(H_i),\]
where we note that $\inn_{\mathbf{w}}(\trop(g_i))_{\hat{\mathbf{u}}}$
may equal $\infty$.  If $\gamma = m_i +\trop(g_i)(\mathbf{w}) <
H_i(\mathbf{w})$ then $F'_i \sim F'_{i+1}$ is equal to
$\inn_{\mathbf{w}}(\trop(g_i)) \sim
\inn_{\mathbf{w}}(\trop(g_i))_{\hat{\mathbf{u}}}$.  In all cases,
Lemma~\ref{l:initialformtrop} ensures that the relation $F'_i \sim
F'_{i+1}$ is in $\Trop(\inn_{\mathbf{w}}(I))$.  Now, suppose that
$F_i(\mathbf{w}) < F_{i+1}(\mathbf{w})$.  If $\gamma =
F_i(\mathbf{w})$ then $\trop(g_i)(\mathbf{w})<H_i(\mathbf{w})$ and
$\inn_{\mathbf{w}}(g_i)$ is a monomial. This means that $(F'_i \sim
F'_{i+1}) = (\trop(\inn_{\mathbf{w}}(g_i)) \sim \infty) \in
\Trop(\inn_{\mathbf{w}}(I))$.  Finally, if $\gamma < F_i(\mathbf{w})$
then $F'_i \sim F'_{i+1}$ is the relation $\infty \sim \infty$, which
is in $\Trop(\inn_{\mathbf{w}}(I))$.
\end{proof}

Note that the second condition in Proposition~\ref{p:chaininsupport} was
crucial in this proof.

We are now ready to prove Theorem~\ref{t:mainthm}.  The proof requires
understanding the effect of changes of coordinates on tropical
varieties and congruences.  The group $\GL(n,\mathbb Z)$ acts on $S$
by monomial change of coordinates.  Explicitly, a matrix $A$ sends 
a tropical polynomial $f(\mathbf{x}) = \bigtplus
a_{\mathbf{u}} \ttimes  \mathbf{x}^{\mathbf{u}}$ to $\bigtplus a_{\mathbf{u}} \ttimes
\mathbf{x}^{A\mathbf{u}} = f(A^T\mathbf{x})$.  We write $A
\cdot f$ for this transformed polynomial.  If $J$ is a congruence on
$S$ then $A \cdot J$ is the congruence generated by $\{ A \cdot f \sim
A \cdot g \, : \, f \sim g \in J \}$.  This action is the tropicalization of
the action of $\GL(n,\mathbb Z)$ on $K[x_1^{\pm 1},\dots,x_n^{\pm 1}]$
that sends 
a
monomial $\mathbf{x}^{\mathbf{u}}$ to $\mathbf{x}^{A\mathbf{u}}$.  
Moreover, the action commutes
with tropicalization: We have $\trop(A\cdot f)=A \cdot \trop(f)$.  In particular, 
this implies that if $I$ is an ideal in $K[x_1^{\pm 1},\dots,x_n^{\pm 1}]$ then
$\trop(V(A \cdot I)) = A \cdot \trop(V(I))$; see \cite{TropicalBook}*{Corollary 3.2.13}.

\begin{proof}[Proof of Theorem~\ref{t:mainthm}]
Let $\mathbf{w}$ lie in the relative interior of a maximal cell
$\sigma$ of the tropical variety $\trop(V(I))$, and let $L =
\spann(\mathbf{w}-\mathbf{w}' : \mathbf{w}' \in \sigma)$.  By
\cite{TropicalBook}*{Lemma 3.3.6} we have
$L=\trop(V(\inn_{\mathbf{w}}(I)))$, so Equation~\eqref{eq:dual2}
implies that $L$ can be recovered from the congruence
$\Trop(\inn_{\mathbf{w}}(I)) = \inn_{\mathbf{w}}(\Trop(I))$. This
means that $L$ is determined by $\inn_{\mathbf{w}}(\Trop(I))$, and
thus by $\Trop(I)$.

 After a monomial change of coordinates we may assume that $L =
 \spann(\mathbf{e}_1,\dots,\mathbf{e}_d)$.  By
 \cite{TropicalBook}*{Corollary 2.4.10} the initial ideal
 $\inn_{\mathbf{w}}(I)$ is homogeneous with respect to the $\mathbb
 Z^d$-grading by $\deg(x_i) = \mathbf{e}_i$ for $1 \leq i \leq d$ and
 $\deg(x_i)=\mathbf{0}$ otherwise, so it has a generating set
 $f_1,\dots,f_r$ where $f_i \in \K[x_{d+1}^{\pm 1},\dots,x_n^{\pm
     1}]$.  By \cite{TropicalBook}*{Lemma 3.4.7}, the multiplicity of
 $\sigma$ equals the dimension $\dim_{\K}(R'/(\inn_{\mathbf{w}}(I)
 \cap R'))$, where $R'=\K[x_{d+1}^{\pm 1},\dots,x_n^{\pm 1}]$.  Let
 $\inn_{\mathbf{w}}(I)^h \subset \K[x_0,\dots,x_n]$ be the
 homogenization of the ideal $\inn_{\mathbf{w}}(I) \cap \K[x_1,\dots,
   x_n]$.  Note that since $R'/(\inn_{\mathbf{w}}(I) \cap R')$ is
 zero-dimensional, the Hilbert polynomial of
 $\K[x_0,x_{d+1},\dots,x_n]/(\inn_{\mathbf{w}}(I)^h)$ is equal to the
 constant polynomial $\dim_{\K}(R'/(\inn_{\mathbf{w}}(I) \cap R'))$,
 and thus equals $\mult(\mathbf{w})$.

By \cite{Giansiracusa2}*{Theorem 7.1.6} the Hilbert polynomial of a
homogeneous ideal $J$ can be recovered from its tropicalization
$\Trop(J) \subset \tilde{S}$, so to show that $\mult(\mathbf{w})$ can
be recovered from $\Trop(I)$ it is enough to show that
$\Trop(\inn_{\mathbf{w}}(I)^h)$ can be recovered from $\Trop(I)$.  By
Proposition~\ref{p:homogenizecongruence} we have
$\Trop(\inn_{\mathbf{w}}(I)^h) = \Trop(\inn_{\mathbf{w}}(I))^h$, and
by Proposition~\ref{p:inicommuteswithtrop} we have
$\Trop(\inn_{\mathbf{w}}(I))^h = \inn_{\mathbf{w}}(\Trop(I))^h$,
so the result follows.
\end{proof}

We can thus recover the tropical cycle from the tropical scheme.  This
can be considered as a tropicalization of the Hilbert-Chow morphism
that takes a scheme to the underlying cycle.

\section{Tropical schemes and valuated matroids}
\label{s:matroids}

In this section we investigate in more depth the structure of the
equivalence classes of $\Trop(I)$.  We restrict our attention 
to the case where $I$ is a homogeneous ideal in the polynomial
ring $K[x_0,\dots,x_n]$; an understanding in this case extends to
ideals in $K[x_1^{\pm 1},\dots,x_n^{\pm 1}]$
using Proposition~\ref{p:homogenizecongruence}.  We prove that any
homogeneous tropical polynomial $F \in \tilde{S}$  has a
distinguished representative in its equivalence class, and we give a
computationally tractable description of it. 
The combinatorial machinery that naturally keeps track of the
information contained in the congruence $\Trop(I)$ is that of
\emph{valuated matroids}.

Valuated matroids are a generalization of the notion of matroids that
were introduced by Dress and Wenzel in \cite{DressWenzel}.  Our sign
convention is, however, the opposite of theirs.  For basics of
standard matroids, see, for example, \cite{Oxley}.

Let $E$ be a finite set, and let $r \in \mathbb N$. 
Denote by $\binom{E}{r}$ the collection of subsets of $E$ of size $r$. 
A {\em valuated
  matroid} $\M$ on the ground set $E$ is a function $p \colon
{\binom{E}{r}} \rightarrow \Rbar$ satisfying the following
properties.
\begin{enumerate}
\item There exists $B \in \binom{E}{r}$ such that $p(B) \neq \infty$. 
\item {\it Tropical Pl\"ucker relations:} For every $B, B' \in \binom{E}{r}$ and every 
$\mathbf{u} \in B - B'$ there exists $\mathbf{v} \in B' - B$ with 
\[p(B) + p(B') \geq p(B - \mathbf{u} \cup 
\mathbf{v}) + p(B' - \mathbf{v} \cup \mathbf{u}).\]
\end{enumerate}
The support 
\[{\textstyle \supp(p) := \{ B \in \binom{E}{r} : p(B) \neq
  \infty \} }\] is the collection of bases of a rank $r$ matroid on the ground set $E$, 
called the {\em underlying matroid} $\underline{\M}$ of $\M$.
The function $p$ is called the {\em basis valuation function} of $\M$.
We consider the basis valuation functions $p$ and $\lambda + p$ for $\lambda \in \mathbb{R}$ to be the same
valuated matroid.

We denote by $\mon_d$ the set of monomials of degree $d$ in the
variables $x_0, \dotsc , x_n$.
Any homogeneous polynomial $f \in
K[x_0,\dots,x_n]$ of degree $d$ can be regarded as a linear form $l_f$ on 
the 
$K$-vector-space $V_d$ with basis $\mon_d$.  Let $I_d$ be the degree
$d$ part of the ideal $I$.  Consider the linear subspace
\[
 L_d := \{ \mathbf{y} \in V_d : l_f(\mathbf{y}) = 0 \text{ for all } f \in I_d \} \subset V_d.
\]
Under the pairing $\langle \cdot , \cdot \rangle : K[x_0,\dots,x_n]_d
\times V_d \to K$ defined by $\langle f,\mathbf{y} \rangle := l_f(\mathbf{y})$, the
linear subspace $L_d$ is orthogonal to $I_d$.  Let $r_d =
\dim(L_d)$.  The linear subspace $L_d$ determines a point in the
Grassmannian $\Gr(r_d, V_d)$.  The coordinates of this point in the
Pl\"ucker embedding of $\Gr(r_d,V_d)$ into $\mathbb P^N$, where $N :=
{\binom{|\mon_d|}{r_d}}-1$, are called the Pl\"ucker coordinates of
$L_d$ (and dually of $I_d$).  They are indexed by subsets of $\mon_d$ of size
$r_d$.

\begin{definition} \label{d:valuatedmatroid}
The valuated
matroid $\M(I_d)$ of $I_d$ is the function $p_d : {\binom{\mon_d}{r_d}} 
\rightarrow \Rbar$ given by setting $p_d(B)$ 
to be the valuation of the Pl\"ucker
coordinate of $L_d \in \Gr(r_d, {\binom{n +d}{d}})$ indexed by
$B$.  We denote by $\underline{\M}(I_d)$ the underlying
matroid of $\M(I_d)$.
\end{definition}

A valuated matroid that comes from taking the valuation of Pl\"ucker
coordinates is called {\em realizable}.  The function $p_d$ is the
{\em tropical Pl\"ucker vector} associated to the tropical linear
space $\trop(L_d)$; it completely determines $\trop(L_d)$
\cite{SpeyerSturmfels}*{Theorem 3.8}. 
The valuated matroids that arise in the tropicalization of an ideal
are all realizable, but we will not need that fact in the proofs in
this section.

Usual matroids have several different ``cryptomorphic'' definitions,
and the same is true for valuated matroids.  In the underlying matroid
$\underline{\M}(I_d)$, a subset of monomials $A \subset \mon_d$ is
dependent if and only if there exists a polynomial $h \in I_d$ with
$\supp(h) \subset A$.  Thus $C \subset \mon_d$ is a {\em circuit} of
$\underline{\M}(I_d)$ if and only if $C= \supp(h)$ for some $h \in
I_d$ of minimal support.  A tropical polynomial $H \in \trop(I)_d$ is
called a {\em vector} of the valuated matroid $\M(I_d)$.  Such an $H$
has the form $\bigtplus_{i=1}^s a_i \ttimes \trop(f_i)$ for some
$f_i \in I_d$ and $a_i \in \mathbb R$.  Vectors of minimal
support are called \emph{valuated circuits} of $\M(I_d)$.  These all
have the form $H = a \ttimes  \trop(h)$ for some $h \in I_d$ of minimal
support and $a \in \mathbb{R}$. If $H$ and $G$ are valuated
circuits of $\M(I_d)$ with the same support then there exists some $a
\in \mathbb{R}$ such that $H = a \ttimes G$.  The set of vectors and the set
of valuated circuits of $\M(I_d)$ each separately determines $\M(I_d)$;
see \cite{MurotaTamura}*{Theorem 3.3}.

With these definitions in place, we can now finish the proof of
Theorem~\ref{t:equivalent}.  We restate it in a slightly generalized
form, allowing more general projective schemes.

\begin{theorem} \label{t:generalizedEquivalence}
Let $Z \subset \mathbb P^n$ be a subscheme defined by a homogeneous
ideal $I \subset K[x_0,\dots,x_n]$.  Then any of the following three objects determines the others:

\begin{enumerate}
\item \label{item:congruence2} The congruence $\Trop(I)$ on  $\tilde{S}$;
\item \label{item:ideal2} The ideal $\trop(I)$ in $\tilde{S}$;
\item \label{item:valuated2} The set of valuated matroids $\{\M(I_d)\}_{d \geq 0}$,
where   $I_d$ is the degree $d$ part of $I$.
\end{enumerate}

Thus if $Y \subset T \cong (K^*)^n$ is a subscheme given by an ideal
$I \subset K[x_1^{\pm 1},\dots,x_n^{\pm 1}]$, and $I^h \subset
K[x_0,\dots,x_n]$ is the ideal of the projective closure $\overline{Y}
\subset \mathbb P^n$ of $Y$, then the ideal $\trop(I) \subset S$ and
the set of valuated matroids $\M(I_d^h)$ for $d \geq 0$ determine each
other.
\end{theorem}

\begin{proof}
The proof that  (\ref{item:congruence2}) determines 
(\ref{item:ideal2}) given at the end of Section~\ref{s:preliminaries}
included the proof for general homogeneous ideals, as we never used that
$I^h$ was a homogenization.  The proof given there that
(\ref{item:ideal2}) determines (\ref{item:congruence2}) is also valid
for homogeneous ideals.

The elements of $\trop(I)_d$ are the vectors of the valuated matroid
$\M(I_d)$, so $\trop(I)$ determines and is determined by the set of
valuated matroids $\{\M(I_d)\}_{d \geq 0}$.  This shows
(\ref{item:ideal2}) $\Leftrightarrow$ (\ref{item:valuated2}).

When $I \subset K[x_1^{\pm 1},\dots,x_n^{\pm 1}]$, the ideal
$\trop(I^h)$ in $\tilde{S}$ is the homogenization of the ideal
$\trop(I)$ in $S$, and also $\trop(I) = \trop(I^h)|_{x_0=0}$. This
shows that $\trop(I)$ determines $\trop(I^h)$ and conversely, so the
last part follows from the first.
\end{proof}

We now investigate in more depth the structure of the equivalence
classes of $\Trop(I)$. 
In what follows, for any homogeneous tropical
polynomial $F \in \tilde{S}_d$ and any $\mathbf{u} \in \mon_d$, we denote
by $F^\mathbf{u}$ the coefficient of $F$ corresponding to the monomial
$\mathbf{u}$.  For $F, G \in \tilde{S}_d$, 
we say that $F \leq G$ if the inequality holds
coefficient-wise, so $F^\mathbf{u} \leq {G}^\mathbf{u}$ for
all $\mathbf{u} \in M_d$.

We will restrict our attention to the case where the subspace $I_d$
does not contain any monomials, so the matroid $\underline{\M}(I_d)$ is a
loopless matroid.  When $I$ does contain a monomial $g = a
\mathbf{x}^{\mathbf{u}}$, the congruence $\Trop(I)$ contains the
relation $\trop(g) \sim \trop(g)_{\hat{\mathbf{u}}} = \infty$.  For
any tropical polynomial $P \in \tilde{S}_d$ and any $\lambda \in
\mathbb{R}$, the relation $ P \tplus \lambda \ttimes
\mathbf{x}^{\mathbf{u}} \sim P$ is then in $\Trop(I)$.  This implies
that the equivalence class of a tropical polynomial $F \in
\tilde{S}_d$ does not depend on the coefficient of the monomial
$\mathbf{u}$, so we do not lose information by ignoring this
coefficient.

Let $F = \bigtplus_{\mathbf{u} \in \mon_d} F^\mathbf{u} \ttimes  \mathbf{x}^{\mathbf{u}}$
be a homogeneous tropical polynomial of degree $d$.
For any circuit 
$C \subset \mon_d$
of $\underline{\M}(I_d)$ and any
$\mathbf{u} \in C$, let $G_{C,\mathbf{u}}$ be the valuated circuit
of $\M(I_d)$ satisfying
$\supp(G_{C,\mathbf{u}}) = C$ and $G_{C,\mathbf{u}}^\mathbf{u} = 0$.
Furthermore, let 
\begin{equation}\label{eq:lambdas}
\lambda_{C,\mathbf{u}} := \max_{\mathbf{v} \in C-\mathbf{u}} \, 
(F^\mathbf{v} - G_{C,\mathbf{u}}^\mathbf{v}) \in \Rbar.
\end{equation}
The subtraction here is in usual arithmetic, where we follow the
convention that $\infty - a = \infty$ for $a \in \mathbb R$.  Since
$\mathbf{v} \in C$ we have $G^{\mathbf{v}}_{C,\mathbf{u}} < \infty$.
The assumption that $\underline{\M}(I_d)$ is loopless ensures that
this maximum is over a nonempty set, so $\lambda_{C,\mathbf{u}} \in
\Rbar$.  Equivalently, $\lambda_{C,\mathbf{u}}$ satisfies
\begin{equation}\label{eq:lambdas2}
\lambda_{C,\mathbf{u}} = \min \left( \lambda \in \Rbar : 
\lambda + (G_{C,\mathbf{u}})_{\hat{\mathbf{u}}} \geq F \right).
\end{equation}
We define the tropical polynomial $\pi(F) \in \tilde{S}_d$ to be the tropical sum 
\[
\pi(F)  :=  F  \tplus \Bigl( \bigtplus_{{
    \mathbf{u} \in C \subset \mon_d}} \lambda_{C,\mathbf{u}}\ttimes G_{C,\mathbf{u}} \Bigr)
 \]
where the inner sum is taken over all circuits $C$ of 
$\underline{\M}(I_d)$ 
and all $\mathbf{u} \in C$.
The coefficient of $\mathbf{v}$ in $\pi(F)$ is 
\begin{equation}\label{eqtn:piFv}
\pi(F)^{\mathbf{v}} = \min \Bigl( \, F^{\mathbf{v}} \,, \, 
\min_{{
    \mathbf{v} \in C \subset \mon_d}} (\lambda_{C,\mathbf{v}} ) \,\Bigr),
\end{equation}
where the inner minimum is only over those circuits $C$ containing
$\mathbf{v}$.

\begin{example}
Consider the ideal $I = \langle x + y + tz, x+y+t^2w \rangle$ in
$\PuiseuxC[x,y,z,w]$.  The underlying matroid $\underline{\M}(I_1)$ in
degree one has ground set $\mon_1 = \{x,y,z,w\}$, and circuits
$\{x,y,z\}$, $\{x,y,w\}$, and $\{z,w\}$. The valuated matroid
$\M(I_1)$ has valuated circuits $x \tplus y \tplus 1\ttimes z$, $x
\tplus y \tplus 2 \ttimes w$, and $z \tplus  1 \ttimes w$.  Consider the
tropical polynomial $F =  x \tplus 1 \ttimes y \in \Rbar[x,y,z,w]$. The
polynomial $\pi(F)$ is equal to
\begin{align*}
 \pi(F) &=   F \tplus  1 \ttimes (x \tplus y \tplus 1 \ttimes z)  \tplus   1  \ttimes (x \tplus y \tplus 2 \ttimes w) \tplus  \infty \ttimes  (z \tplus  1 \ttimes w) \\
 &= x\tplus  1 \ttimes y \tplus  2 \ttimes z \tplus 3 \ttimes w.
\end{align*}
Similarly, for the tropical polynomial $F' = 2 \ttimes w \in S$ we have
\begin{align*}
 \pi(F') &=  F'  \tplus  \infty  \ttimes  (x \tplus y \tplus 1 \ttimes z) \tplus  \infty  \ttimes  (x \tplus y \tplus 2 \ttimes w) \tplus 1  \ttimes  (z\tplus  1 \ttimes w) \\
 &= 1 \ttimes z\tplus  2\ttimes w. \tag*{\qedhere} 
\end{align*}
\end{example}

The following proposition shows that $\pi(F)$ is the coefficient-wise
smallest tropical polynomial in the equivalence class of $F$ in
$\Trop(I)$.  It is thus a distinguished representative of the
equivalence class.
\begin{proposition}\label{prop:properties}
The map $\pi : \tilde{S}_d \to \tilde{S}_d$ satisfies the following properties:
\begin{enumerate}[label={\normalfont (\alph*)}]
\item \label{it:less} $\pi(F) \leq F$.
\item \label{it:proj} $\pi(\pi(F)) = \pi(F)$.
\item \label{it:sim} $F \sim \pi(F) \, \in \Trop(I)$.
\item \label{it:unique} $F \sim F' \, \in \Trop(I) \, \Longleftrightarrow \, \pi(F) = \pi(F')$.
\end{enumerate}
\end{proposition}

In the proof of Proposition \ref{prop:properties} we will make use of the following facts about valuated circuits:
\begin{enumerate}
\item \label{enum:VCfact1}
 If $H$ is a vector of $\M(I_d)$ with $\mathbf{u}
  \in \supp(H)$ then there is a valuated circuit $G$
  with $G^{\mathbf{u}} =
  H^{\mathbf{u}}$ and $G \geq H$.

\item \label{enum:VCfact2} If $H$ is a vector and $G$ is a valuated circuit of
  $\M(I_d)$ with $H^{\mathbf{u}} = G^{\mathbf{u}} < \infty$ and $H^{\mathbf{v}} >
  G^{\mathbf{v}}$, then there is a valuated circuit $G'$ of $\M(I_d)$
  with $G' \geq \min(H,G)$, $G'^{\mathbf{v}} = G^{\mathbf{v}}$, and $G'^{\mathbf{u}} = \infty$.
\end{enumerate}

Fact~\eqref{enum:VCfact1} follows from Theorems 3.4 and 3.5 of
\cite{MurotaTamura} and the definition given there of the function
$\phi_{\mathcal X \rightarrow \mathcal V}(\mathcal X)$.
Fact~\eqref{enum:VCfact2} is a combination of
Fact~\eqref{enum:VCfact1} and the valuated circuit elimination axiom
\cite{MurotaTamura}*{Theorem 3.1 (VCE)}.

\begin{proof}[Proof of Proposition \ref{prop:properties}]
Property \ref{it:less} follows directly from the definition, since $F$
is a tropical summand of $\pi(F)$.  Property \ref{it:proj} follows
from properties \ref{it:sim} and \ref{it:unique}, which we now prove.  In
order to show that Property \ref{it:sim} holds, fix an enumeration
$\{(\mathbf{u}_1,C_1), \dotsc, (\mathbf{u}_s, C_s)\}$ of the set
$\{(\mathbf{u},C): C \text{ is a circuit of $\underline{\M}(I_d)$}
\text{ and } \mathbf{u} \in C\}$.
For $0 \leq i \leq s$, set
\[
H_i :=  F \tplus \Bigl( \bigtplus_{1 \leq j \leq i}
\lambda_{C_j,\mathbf{u}_j} \ttimes G_{C_j,\mathbf{u}_j} \Bigr)
\]
so that $H_0 = F$ and $H_s = \pi(F)$. 
By Equation \eqref{eq:lambdas2}, for any $i$ we have
$H_{i-1} \leq F \leq \lambda_{C_i,\mathbf{u}_i} \ttimes  (G_{C_i,\mathbf{u}_i})_{\hat{\mathbf{u}}_i}$.
Since $\Trop(I)$ is a congruence, the relation 
\[
H_{i-1} = H_{i-1} \tplus  \lambda_{C_i,\mathbf{u}_i} \ttimes (G_{C_i,\mathbf{u}_i})_{\hat{\mathbf{u}}_i}  \sim 
H_{i-1}\tplus  \lambda_{C_i,\mathbf{u}_i} \ttimes G_{C_i,\mathbf{u}_i} = H_i
\]
is in $\Trop(I)$.  The result follows from transitivity.

We now prove Property \ref{it:unique}. If $\pi(F) = \pi(F')$ then by Property \ref{it:sim} we have 
$F \sim \pi(F) = \pi(F') \sim F'$, so $F \sim F'$. 
In order to prove the converse statement, by Lemma
\ref{lem:semimodule} it is enough to show that $\pi(H \tplus P) =
\pi(H_{\hat{\mathbf{u}}} \tplus P)$ for any vector $H$ 
of $\M(I_d)$,
$\mathbf{u} \in \supp(H)$, and $P \in \tilde{S}_d$.  Set
\[
F := H \tplus P \quad \text{ and } \quad F' := H_{\hat{\mathbf{u}}} \tplus P.
\] 
Note that $F$ and $F'$ can only differ in the coefficient corresponding to
the monomial $\mathbf{u}$. We will assume that $F^\mathbf{u}
= H^\mathbf{u} < P^{\mathbf{u}}$, as otherwise $F = F'$. 

For any circuit $C$ 
of $\underline{\M}(I_d)$
and any $\mathbf{u}' \in C$, let
$\lambda_{C,\mathbf{u}'} \in \Rbar$ be as in Equation
\eqref{eq:lambdas}.  Let $\lambda'_{C,\mathbf{u}'}$ be defined
analogously for the tropical polynomial $F'$.  Since $F \leq F'$, we
have $\lambda_{C,\mathbf{u}'} \leq \lambda'_{C,\mathbf{u}'}$.  It
follows that $\pi(F) \leq \pi(F')$.
Since $F^{\mathbf{v}} = {F'}^{\mathbf{v}}$
for $\mathbf{v} \neq \mathbf{u}$, we see from Equation~\eqref{eq:lambdas} 
that $\lambda_{C,\mathbf{u}} =
\lambda'_{C,\mathbf{u}}$.  Thus Equation~\eqref{eqtn:piFv}
implies that $\pi(F)^{\mathbf{u}} = \pi(F')^{\mathbf{u}}$.

Suppose that $\pi(F)^{\mathbf{v}} < \pi(F')^{\mathbf{v}}$ for some
$\mathbf{v}\neq \mathbf{u}$.  By Equation~\eqref{eqtn:piFv}, there must be a circuit $C$ 
of $\underline{\M}(I_d)$ with $\mathbf{v} \in C$ and both
$\lambda_{C,\mathbf{v}} < \pi(F')^{\mathbf{v}} \leq
\lambda'_{C,\mathbf{v}}$ and $\lambda_{C, \mathbf{v}} < F^{\mathbf{v}}
= {F'}^{\mathbf{v}} \leq H^{\mathbf{v}}$.  The maximum in
Equation~\eqref{eq:lambdas} must then be achieved at the coefficient
of $\mathbf{u}$, as this 
is 
the only coefficient for which $F$ and $F'$
differ,
so $\lambda_{C,\mathbf{v}} = F^{\mathbf{u}} -
G_{C,\mathbf{v}}^{\mathbf{u}}$.
 Note that this implies in particular that $\mathbf{u} \in C$.  By
 Fact~\eqref{enum:VCfact2} applied to the vector $H$ and the valuated
 circuit $\lambda_{C,\mathbf{v}} \ttimes G_{C,\mathbf{v}}$,
there is a valuated circuit $G'$ of support $C'$ with $\mathbf{u}
\not \in C'$, ${G'}^{\mathbf{v}} = \lambda_{C,\mathbf{v}}$, and
$G' \geq \lambda_{C,\mathbf{v}} \ttimes G_{C,\mathbf{v}} \tplus H$. 
The valuated circuit $G'$ must then be equal to 
$\lambda_{C,\mathbf{v}} \ttimes G_{C',\mathbf{v}}$. 
We now have
\[\lambda_{C,\mathbf{v}} \ttimes (G_{C',\mathbf{v}})_{\hat{\mathbf{v}}} = G'_{\hat{\mathbf{v}}}  \geq
\lambda_{C,\mathbf{v}} \ttimes (G_{C,\mathbf{v}})_{\hat{\mathbf{v}}} \tplus
H_{\hat{\mathbf{v}}} \geq F, 
\]
and thus in view of Equation \eqref{eq:lambdas2}, $\lambda_{C',\mathbf{v}} \leq
\lambda_{C,\mathbf{v}}$.  Since $\mathbf{u} \not \in C'$, we
have $\lambda_{C',\mathbf{v}} = \lambda'_{C',\mathbf{v}}$ by
Equation~\eqref{eq:lambdas}, which contradicts
$\lambda_{C,\mathbf{v}} < \pi(F')^{\mathbf{v}} \leq
\lambda'_{C',\mathbf{v}}$.  We thus conclude that
$\pi(F)^{\mathbf{v}} = \pi(F')^{\mathbf{v}}$ for all $\mathbf{v}$.
\end{proof}

\begin{remark}
Note that we have $\lambda_{C,\mathbf{u}} < \infty$ if and only if $C - \mathbf{u} \subset \supp(F)$. 
It follows that $\supp(\pi(F))$ equals the closure $E_F$ of
the set $\supp(F)$ in the matroid $\underline{\M}(I_d)$.
In fact, we may regard $\pi(F)$ as being the ``valuated closure'' of $F$ in the valuated matroid
$\M(I_d)$.  
The definition of $\pi$ makes sense for any valuated matroid $\M$, 
and its properties stated in Proposition \ref{prop:properties} 
remain valid in this more general setup.
It would be interesting to develop a set of cryptomorphic axioms for valuated matroids
from this ``valuated closure'' perspective.
\end{remark}

\begin{remark}
The map $\pi$ can also be thought of as a tropical projection, in the following sense. 
One can extend the function $p_d$ to all subsets of $\mon_d$, 
obtaining in this way a valuation function for all independent subsets of $\underline{\M}(I_d)$ \cite{Murota}.
Concretely, for any $A \subset \mon_d$ define
\begin{equation*}\label{eq:indep}
{\textstyle p_d(A) := \min \{ p_d(B) : A \subset B \in \binom{\mon_d}{r_d}\},}
\end{equation*}
with the convention that $p_d(A) = \infty$ if the corresponding set is empty.
We have $p_d(A) \neq \infty$ if and only if $A$ is an independent set of $\underline{\M}(I_d)$.
Given any subset $E \subset \mon_d$, the restriction of the function $p_d$ to 
the maximal independent subsets of $E$ gives rise to a valuated matroid on the set $E$, called the
restriction $\M(I_d)|E$ of $\M(I_d)$ to $E$.

Now, suppose $F$ is a homogeneous tropical polynomial of degree $d$.
Let $E_F$ be the closure of
$\supp(F)$ in the matroid $\underline{\M}(I_d)$.
If $\supp(F) = E_F$, it follows from Equation \eqref{eqtn:piFv} and 
\cite{Corel}*{Section 4} that $\pi(F)$ is the tropical
projection of $F \in \Rbar^{E_F}$ onto the tropical linear space in $\Rbar^{E_F}$
corresponding to the valuated matroid $\underline{\M}(I_d)|E_F$,
{\em but taking tropical sum to be $\max$ instead of of $\min$}.
If $\supp(F) \subsetneq E_F$ then we have to be more careful: The 
tropical polynomial $\pi(F)$ is the tropical projection of $F$
after substituting the coefficients corresponding to monomials in 
$E_F - \supp(F)$ by large enough real numbers.

Tropical projections onto tropical linear spaces have been studied
in \cites{Ardila, Corel, Rincon2}. Using those results one can obtain 
a description of $\pi(F)$ 
amenable to computational purposes, as we now describe.
For any basis $B$ of $E_F$ (i.e., a maximal independent set contained in $E_F$),
define 
\[
w_{F}(B) := p_d(B) + {\textstyle \sum_{\mathbf{u} \in B} F^\mathbf{u} \in \Rbar}.
\]
Let $B_F$ be a basis of $E_F$ such that 
$w_{F}(B_F)$ is minimal among all bases $B_F$ of $E_F$.
Note that the value of $w_{F}(B_F)$ is finite, so $B_F \subset \supp(F)$.
For any 
$\mathbf{u} \in E_F - B_F$ 
there exists a unique circuit $C(B_F,\mathbf{u})$ of $\underline{\M}(I_d)$ 
contained in $B_F \cup \mathbf{u}$,
called the \emph{fundamental circuit} of $\mathbf{u}$ over $B_F$.
It is equal to
\[
 C(B_F,\mathbf{u}) = \{ \mathbf{v} \in B_F : B_F \cup \mathbf{u} - \mathbf{v} \text{ is independent in } \underline{\M}(I_d) \} \cup \mathbf{u}.  
\]
With this notation in place, \cite{Corel}*{Section 4, Proposition 5} implies that the coefficients of $\pi(F)$ are given by
\begin{equation*}
 \pi(F)^\mathbf{u} = \begin{cases}
  F^\mathbf{u} & \text{if } \mathbf{u} \in B_F,\\
  {\displaystyle \max_{\mathbf{v} \in C(B_F,\mathbf{u})-\mathbf{u}} \, \left(F^\mathbf{v} - p_d(B_F \cup \mathbf{u} - \mathbf{v}) + p_d(B_F) \right) } & \text{if } \mathbf{u} \in E_F - B_F,\\
 \infty & \text{if } \mathbf{u} \notin E_F.
\end{cases}
\end{equation*}
The computation of the coefficients $\pi(F)^\mathbf{u}$ using this
description involves computing a maximum over only one circuit of
$\underline{\M}(I_d)$. This makes it computationally much simpler than
formula \eqref{eqtn:piFv}, assuming that we know the function $p_d$.
\end{remark}

\end{document}